\documentclass[
10pt,
letterpaper,
oneside,
headinclude,footinclude, 
BCOR=5mm, 
]{article}

\usepackage{mathptmx}
\usepackage{helvet}
\usepackage{courier}
\usepackage{type1cm}         

\usepackage{makeidx}         % allows index generation
\usepackage{graphicx}        % standard LaTeX graphics tool
\usepackage[bottom]{footmisc}% places footnotes at page bottom

\usepackage{dcolumn}
\newcolumntype{d}[1]{D{.}{.}{#1}}
\usepackage{amsthm}
\usepackage{mathrsfs}
\usepackage{amsmath}
\usepackage{amssymb}
\usepackage{newtxtext}       % 
\usepackage{newtxmath}       % selects Times Roman as basic font
\usepackage{epigraph}
\usepackage[LGR,T1]{fontenc}
\usepackage{accents}
\usepackage{bbding}
\usepackage{lscape}

\usepackage{hyperref}

\title{\normalfont{Post-Einsteinian Effects in the General Theory of Relativity from Higher-Order Riemannian Geometry}} 

\author{{William Edward Bies
} \\ {E-mail: william.bies.phd@gmail.com
}}

\date{June 6, 2024} 

\begin{document}
	
	\numberwithin{equation}{section}
	\theoremstyle{plain}
	\newtheorem{theorem}{Theorem}[section]
	\newtheorem{proposition}{Proposition}[section]
	\newtheorem{lemma}{Lemma}[section]
	\newtheorem{corollary}{Corollary}[section]
	\newtheorem{postulate}{Postulate}[section]
	
	\theoremstyle{definition}
	\newtheorem{definition}{Definition}[section]
	
	\theoremstyle{remark}
	\newtheorem{remark}{Remark}[section]
	
	\renewcommand{\qedsymbol}{\small \RectangleBold}
	
	\maketitle
	
	\setcounter{tocdepth}{2} 
	
	\tableofcontents
	
	\newpage
	
	\section*{Abstract}
	
In Part I of this series, the author has shown how to extend the framework of Riemannian geometry so as to include infinitesimals of higher than first order. The purpose of the present contribution is to initiate an investigation into the implications of higher-order differential geometry for the general theory of relativity. As we have seen, a novel concept of inertial motion is implied by the analogue of the geodesic equation when modified to include the effects of the higher infinitesimals and it therefore should not come as a surprise that it has potentially observable kinematic consequences. The route we prefer to take goes through the Einstein-Hilbert action, generalized to reflect the presence of higher infinitesimals and a cosmological constant. A variational principle yields an hierarchy of field equations, which reduce to the Einsteinian case at first order. In the weak-field limit, we recover the usual relativistic equation for a moving body to leading order. To exemplify the theoretical framework, we undertake a preliminary study of the Schwarzschild solution and Friedmann-Robertson-Walker cosmology in the presence of second-order terms. But our most exciting results concern the novel effects in orbital mechanics that arise when the higher-order corrections cannot be neglected. Indeed, the higher-order Riemannian geometry predicts a modification of Newtonian dynamics corresponding to the Pioneer anomaly for a spacecraft on a hyperbolic escape trajectory from the solar system and to the flyby anomaly for the differential between ingoing and outgoing asymptotic velocities of a spacecraft passing near a rotating planet---both of which are found to agree well with sensitive empirical findings.

	\vspace{12 pt}

2020 Mathematics Subject Classification. Primary: 53A45 Differential geometric aspects in vector and tensor analysis. Secondary: 83C05 Einstein's equations (general structure, canonical formalism, Cauchy problem); 83C10 Equations of motion in general relativity and gravitational theory; 83B05 Observational and experimental questions in relativity and gravitational theory

\vspace{12 pt}

Keywords: higher-order differential geometry, general theory of relativity, experimental tests of general relativity 

\vspace{12 pt}

Subject Classifications: Riemannian geometry, local differential geometry, analysis on manifolds

\vspace{12 pt}

%Declarations of interest: None. This research did not receive any specific grant from funding agencies in the public, commercial, or not-for-profit sectors.	

	\section{Introduction}\label{chapter_1}

In Part I of the present series of papers, we propose a consistent Riemannian geometry incorporating infinitesimals to higher than first order and derive its elementary properties. Next, in Parts II and III, we take up the application of these ideas to the general theory of relativity. As ordinarily construed, the general theory of relativity concerns the problem of gravitation among bodies arranged across macroscopic distances. In this context, we may expect to see the novel inertial effects of the higher-order contributions to the jet geodesic equation arising from the Levi-Civita connection reflecting the generalized Riemannian metric to manifest themselves over sufficiently large space-time intervals. The problem, thus, becomes twofold: first, to devise a suitable form of the Einstein-Hilbert action in the current setting that will lead to field equations from which a relativistic equation of motion to higher order can be derived; and second, to investigate the observational consequences of the formalism so developed for the motion of orbiting bodies in the solar system. These will form our objective in {\S}\ref{chapter_6} and {\S}\ref{chapter_7} below, respectively, which together make up Part II. 

Yet the theory advanced in Part I is quite versatile and we are far from having exhausted its consequences for fundamental physics. For the generalization of the Riemannian curvature tensor as there defined entails a reconceptualization not only of space-time structure at large scales, but also in the domain of the infinitely small. The novel possibilities thereby suggested will become perspicuous once the geometrical situation is viewed in terms of the Cartan formalism, from which it immediately follows that the affine connection 1-forms must acquire non-trivial higher-order components, the phenomenology of which will be to give rise to novel interactions among bodies across large distances by means of the additional newly recognized field degrees of freedom. The key point to realize now is that the form assumed by these interactions will be \textit{necessarily} and \textit{uniquely} dictated by the structure of Minkowskian space-time in the infinitely small. The resulting constructive principles will provide a cogent ground for a dynamical philosophy of the fundamental forces acting in unison which is free of speculative mechanistic hypotheses. The detailed investigation of this problem will form the subject of Part III.

\subsection{Motivations from Theoretical Physics}

It is worth pointing our here how our methodology differs from that of other prominent research programs current in the literature: we adopt a dynamical rather than mechanistic philosophy as the proper basis of a science of nature (for a good entryway into the pertinent issues, see Hendry's study of Maxwell in \cite{hendry}). For our purposes, the import of such a decision amounts, namely, to this, that we prefer to depart from feigning of hypotheses in order to embrace a constructive method. The problem with the hypothetico-deductive method as generally practiced today is that it encourages eclecticism, or a random shooting into the dark when it comes to the formulation of the mechanical model from which everything else is to be derived, whereas a constructive method favors, instead, a systematic, inductive inference to the governing physical principles, whence it tends to be informed by a greater degree of connection with experiment. The failure of grand unified theories and superstring theory to yield any insight concerning the peculiar and seemingly problematical features of the standard model---the apparent contingency behind the selection of gauge groups and fermionic representations and the large number of free parameters---may be attributed to just this defect of method. A dynamical approach, on the other hand, will be characterized by the necessity and uniqueness that attend its constructions and therefore issues in \textit{testable} predictions.

Another motivation that will turn up in the course of our investigations is the following. As we shall see in due course, a dynamical philosophy of nature inspires an actional perspective as being the most efficacious for bringing about the transfer from general principles to specific mathematical formulae, along the lines of what Maxwell does for electrodynamics \cite{maxwell_treatise}. In our case, the procedure that suggests itself is to seek a reformulation of the Einstein-Hilbert action capable of incorporating the higher infinitesimals. When one pursues this possibility, it turns out by serendipity that a cosmological term can be introduced in such a way as simply to account for the dark energy that is responsible for the acceleration of the rate of expansion of the universe (for reviews, see Peebles and Ratra \cite{peebles_ratra} and Frieman \textit{et al.} \cite{frieman_turner_huterer}). It proves expedient to view dark energy as a background contribution to the generalized Riemannian curvature 2-forms in Cartan's formalism. This reconceptualization of the role of dark energy enables an auspicious reformulation of the conventional Einstein-Hilbert action in such a way as to make it conspicuous how it can embrace at the same time gravitation and interactions mediated by gauge fields of Yang-Mills type.

	\section{The General Theory of Relativity for Jets}\label{chapter_6}

As one might well expect, the circle of ideas so far put forward in Part I ought to have far-reaching applications to physics. Here in {\S}\ref{chapter_6}, we advance a reformulation of Einstein’s general theory of relativity, i.e., in indefinite signature, to include the effects of the higher infinitesimals. Key to everything that follows is a distinction between Cartesian and uniformizing coordinates when representing the Minkowskian metric in free space, the physical significance of which we discuss. Then, a field equation involving a hierarchy of sectors at higher order in the jets is obtained from a variational principle. The weak-field limit is derived and the relativistic equation of motion for a planet in the solar system is recovered in a suitable limit. We conclude with a cursory discussion of the Schwarzschild solution and cosmological solutions of Friedmann-Robertson-Walker type in order to exhibit the leading corrections due to jets of higher than first order.

\subsection{From Pure Mathematics to Theoretical Physics}

Despite its abundant successes in explaining the phenomena of gravitational physics, Einstein’s general theory of relativity \cite{einstein_grundlage} leaves a number of pressing questions unresolved. The two foremost about which one wonders are first, how does gravity relate to the other fundamental forces and second, is it possible to quantize it along lines similar to what has been done so admirably with respect to the standard model of the electroweak and strong forces? In view of the outstanding problems encountered, substantial progress on either front would seem to depend upon as-yet unknown foundational principles. In the present section, we wish to investigate how the general theory of relativity might be reformulated so as to become the physical counterpart to Riemannian geometry at higher order in the infinitesimals, as we have explored above.

To proceed from flat to curved space in {\S}\ref{general_principles} below, we have to recapitulate the steps Einstein takes when originally formulating the theory. While our procedure is analogous to Einstein's, there are still choices to be made. How to extend the Einstein-Hilbert action, how to define the generalized stress-energy tensor on the right hand side of the field equation etc.? A more basic point would be the following: the physical significance of the formal constructions changes because we now allow motion in higher tangent directions. We pursue here in {\S}\ref{chapter_7} and later in Part III the striking ramifications of the resulting coupling of different orders in the infinitesimals.

\subsection{Extension of the General Theory of Relativity to Include Infinitesimals}

After clearing several technical preliminaries out of the way, we will be prepared at last in {\S}\ref{general_principles} to enunciate the postulates that guide our extension of the general theory of relativity. The development is, in some respects, purely formal, but in others, most definitely not so, for our primary in concern in the remainder of this and following sections must be to tease out the \textit{physical} implications of the disarmingly ingenuous mathematical propositions, so easy to state given the formalism now at our disposal and yet so rich in their consequences for physics!

\subsubsection{Minkowskian Metric to Higher Order}\label{Minkowski_tensor_to_higher_order}

To follow the historical course of development, the generally covariant theory of relativity ought to be grounded in the special theory of relativity. What are we to mean by a flat space-time in the presence of infinitesimals? 

The result of proposition I.3.4 enables us to retain the Newtonian and Einsteinian picture of inertia in free space as a leading approximation over small enough scales. Euclidean coordinates can serve as a local model of space-time (where, as is conventional, we use the speed of light in vacuum, $c$, as a constant factor by which to convert time into spatial units). Posit a world in which the bodies are light and far enough apart that their gravitational interactions can be neglected. Introduce the generalized Minkowskian metric $\hat{\eta}$ as follows. Its leading part should agree with the standard form of the special theory of relativity; that is,
\begin{equation}
	\eta_{\mu\nu} = \begin{cases}
		-1 & \text{if $\mu=\nu=0$}, \cr
		\, \, \, 1 & \text{if $\mu=\nu=1,2,3$}, \cr
		\, \, \, 0 & \text{otherwise}.
	\end{cases}
\end{equation}
Therefore, in analogy with equation (II.5.88) we propose the following form:
\begin{equation}\label{generalized_minkowski_metric}
	\hat{\eta} := \sum_{\lambda=1}^\infty
	\sum_{\begin{smallmatrix}\alpha, \beta, \gamma ~\mathrm{such~that}~\alpha\ge\gamma, \beta\ge\gamma, |\gamma|=\lambda \\ \mathrm{either}~ \alpha \le \beta ~\mathrm{or}~ \beta \le \alpha \end{smallmatrix}}  
	\frac{|\alpha|!}{\alpha!}\frac{|\beta|!}{\beta!}
	(-1)^{\beta_0}
	d^\alpha \otimes d^\beta.
\end{equation}

\begin{proposition}
	The postulated form of the Minkowskian metric $\hat{\eta}$ is invariant under Lorentz transformations and, in fact, the full Poincar{\'e} group.
\end{proposition}
\begin{proof}
	Invariance under rigid translations being obvious, all we have to show is the same with respect to a Lorentz transformation, denoted $\Lambda$. Now, just as above in the Euclidean case, $\Lambda$ acts on higher tangents via
	\begin{equation}
		\hat{\Lambda} := \Lambda \oplus \Lambda \otimes \Lambda \oplus \cdots \oplus \overbrace{\Lambda \otimes \cdots \otimes \Lambda}^{r~\mathrm{times}}.
	\end{equation}
	Again, when viewed in the right basis $\hat{\eta}$ is diagonal. The reader may easily check that equation (\ref{generalized_minkowski_metric}) can be succinctly rewritten as
	\begin{equation}\label{minkowski_metric_wrt_coframe_basis}
		\hat{\eta} = \eta \oplus \eta \otimes \eta \oplus \cdots \oplus \overbrace{\eta \otimes \cdots \otimes \eta}^{r ~\mathrm{times}}
	\end{equation}
	with respect to the natural frame defined by the $\varepsilon^{x_{1,\ldots,n}}$ and up to $r$-fold products of these, where $\eta = \varepsilon^\mu \otimes \varepsilon^\mu$ with $\mu$ ranging over 0,1,2,3. Then under the action of the Lorentz transform it becomes
	\begin{align}
		\hat{\Lambda}^t \hat{\eta} \hat{\Lambda} &=
		\left( \Lambda^t \eta \Lambda \right) \oplus \left( \Lambda^t \eta \Lambda \otimes \Lambda^t \eta \Lambda \right) \oplus \cdots \oplus \left( \overbrace{ \Lambda^t \eta \Lambda \otimes \cdots \otimes \Lambda^t \eta \Lambda}^{r ~\mathrm{times}} \right) \\
		&= \eta \oplus \eta \otimes \eta \oplus \cdots \oplus \overbrace{\eta \otimes \cdots \otimes \eta}^{r ~\mathrm{times}} = \hat{\eta}
	\end{align}
	in view of the defining property of the Lorentz transform that $\Lambda^t \eta \Lambda = \eta$. This last step completes the proof.
\end{proof}

\begin{remark}
	We have merely hypothesized the natural form of the Minkowskian metric that would correspond to the Ansatz (I.5.88) in Euclidean space. It would be premature, however, to regard the postulatory structure of the theory as closed. Perhaps, upon deeper reflection, it will become possible to derive it operationally from physical principles, in much the same way as Einstein obtains the Lorentz transformations to first order (and with it, implicitly, the form of the Minkowskian metric to first order as well) from the principles of covariance and the constancy of the speed of light in all inertial frames.
\end{remark}

\subsubsection{Lorentz Transform to Higher Order}\label{lorentz_to_higher_order}

Let us start with a thought experiment concerning inertial motion in presence of higher-order infinitesimals: suppose a body moves in free space i.e. in the absence of any impressed forces. In virtue of Mach's principle as formulated by Einstein in 1918 \cite{einstein_mach_principle}, the body should have no need to know about the disposition of matter elsewhere at great distances away, as somehow its integrated effect will be represented locally through the Minkowskian metric tensor. But, once we admit the presence of higher-order infinitesimals, it is not anymore evident whether or not the body must continue to move with a constant velocity (as this would be measured with respect to the inertial frame established by the distant stars---which we suppose to have been laid out beforehand by means of rigid rods and equipped with stationary clocks by which to read out the time).

There are two conceivable methods of measuring distance, which a priori need not necessarily yield the same result: 1) seriation of rigid rods, where, in the limit each rod is infinitely small compared to the spatial interval to be measured and the distance can be estimated simply by counting the number of rods it takes to fill the gap and 2) measurement of the proper time it takes to go from the initial to final terminus under conditions of inertial motion, where the distance is obtained by integrating velocity with respect to the fixed stars times the increment of proper time. Let $\mu$ be the measure defined in {\S}I.5.1. Then the two measures of arclength from proper time zero up to $\lambda$ work out to be as follows:

\begin{align}
	L_0 &:= \lim_{n \rightarrow \infty} \sum_{m=1}^n \mu \left[ \frac{(m-1)\lambda}{n},\frac{m\lambda}{n} \right] = \lim_{n \rightarrow \infty} n \left( e^{\lambda/n} - 1 \right) = \lambda \\
	L_1 &:= \mu [0,\lambda] = e^\lambda - 1 = \lambda + \frac{1}{2}\lambda^2 + \cdots
\end{align}
The ratio $L_1/L_0$ would tend to unity at short enough times for which $\lambda \ll 1$. Hence, we see that in post-Einsteinian general theory of relativity inertial motion does \textit{not} mean uniform translational velocity with respect to the fixed stars. The two procedures have a consistent interrelationship in that seriation of rigid rods by design eliminates the correction originating in higher differentials and therefore sets up a well defined coordinate system with respect to which the inertial motion can be measured.

With this preliminary, let us now discuss the Lorentz transform to second order in infinitesimals. What the Lorentz transform embodies is the following idea: given a body moving freely in Minkowski space, in the co-moving frame the time coordinate will be just proper time; the problem is to find the transformation from the original inertial frame to the co-moving frame, such that the induced metric on the world-line may become just $- d^\lambda \otimes d^\lambda$. To generalize, follow same procedure but we want to go the natural frame, for which the inertial frame is only the linear approximation valid for times $\Delta t \ll b/c$ where since the cosmological parameter $b$ is presumably large (at least many light-years) we can expect the inertial approximation to hold good for the purpose of all calculations within the solar system to a very high, if not perfect accuracy.

First, we reproduce the usual results in the first-order case. Let $\Lambda$ denote the Lorentz transform to an inertial frame moving with velocity $\mathbf{v}$ with respect to the fixed stars. Then, with $\lambda$ designating arclength along the world-line of a body in uniform translational motion with this velocity, we have coordinates given by the injection $\iota: [0,\lambda] \rightarrow \vvmathbb{R}^4$ as
\begin{equation}
	\lambda \mapsto \left( t_0 + \Lambda^0_0 \lambda, x_0 + \Lambda^0_1 \lambda, y_0 + \Lambda^0_2 \lambda, z_0 + \Lambda^0_3 \lambda \right).
\end{equation}
Now, $\iota_* \partial_\lambda = \Lambda_0^\mu \partial_\mu$. Hence, we can pull back the Minkowskian metric $\eta$ to the world-line as follows:
\begin{equation}
	\iota^* \eta \left( \iota_* \partial_\lambda, \iota_* \partial_\lambda \right) =
	\left( \Lambda_0^\nu \partial_\nu \right)^t \eta \Lambda_0^\mu \partial_\mu =
	\left( \Lambda^t \eta \Lambda \right)_{00} = \eta_{00} = -1.
\end{equation}
Therefore, $\iota^* \eta = - d^\lambda \otimes d^\lambda$, as it ought to. Thus, we arrive at the important result that the kinematics of inertial motion in the generalized theory will reduce to the Einsteinian case in so far as infinitesimals of second and higher order can be neglected. 

To compute the second-order case, we have to find the form of the Lorentz transform to this order. To this end, let $x^{0,\ldots,3}$ denote coordinates in Minkowski space in a frame stationary with respect to the fixed stars and let $y^{0,\ldots,3}$ denote the corresponding natural frame, obtained as
\begin{equation}
	y^\mu = \sum_{\alpha: ~\alpha_\mu \ge 1} x^\alpha.
\end{equation} 
It will be convenient to determine the Lorentz transform with respect to the latter, since the Minkowskian metric becomes diagonal in the natural coordinates. Consider now a proposed coordinate transform of the following form:
\begin{equation}
	y^{\prime\mu} = A^\mu_\nu y^\nu + B^\mu_{\nu\lambda} y^\nu y^\lambda.
\end{equation}
If $\hat{\eta}$ denotes the Minkowskian metric, we want to impose the condition that $\hat{\eta}^\prime = \hat{\eta}$; in other words, invariance up to quadratic terms so that the mapping $y^{0,\ldots,3} \rightarrow y^{\prime 0,\ldots,3}$ may be interpreted as a Lorentz transformation. An important point at this stage of the argument is the following: from the block-diagonal form the the generalized transformation law for jets, equation (I.2.7) above, it is apparent that one can consistently restrict to jets of first order only (or in general, go down from order $r_1$ to order $r_0 < r_1$). It is therefore clear that $A$ has to be a Lorentz transformation, or $A \in SO(1,3)$. What about $B$?

There are two conditions on the $B$ that have to be satisfied: when we expand $d^{y^{\prime}} \otimes \hat{\eta} d^{y^{\prime}}$ we wish to obtain $\eta \otimes \eta$ in the 22-sector and the off-diagonal part in the 12-sector should cancel. We may write them schematically as,
\begin{align}
	\mathrm{tr} ~\mathrm{tr} ~B \otimes \eta \otimes \eta \otimes B &= 0 \\
	\mathrm{tr} ~\mathrm{tr} ~A \otimes (\eta \oplus \eta \otimes \eta) \otimes B &= 0. \label{second_Lorentz_condition}
\end{align}
The second condition is satisfied automatically because the Minkowskian metric is block-diagonal in the natural frame. The first condition says that the 2-jets $B^{0,1,2,3}$ have to be isotropic, or lie on the light cone in the space of 2-jets, and be mutually orthogonal. In other words, we have derived the following proposition:\footnote{The author wishes to thank F. Kuipers for bringing to his attention the incorrect statement of this proposition in the first version of the manuscript (personal communication, December 8, 2023).}
\begin{proposition}\label{second_order_lorentz}
	The most general Lorentz transform at second order assumes the form
	\begin{equation}
		y^{\prime\alpha} = (\Lambda_0)^\alpha_\beta y^\beta + B^\alpha_{\nu\lambda} y^\nu y^\lambda,
	\end{equation}
	for some $\Lambda_0 \in SO(1,3)$ and for any four mutually orthogonal isotropic 2-jets $B^{0,1,2,3}$.
\end{proposition}

\subsubsection{General Principles and Field Equations}\label{general_principles}

Covariance is not necessarily a trivial issue; see Norton's 1992 paper \cite{norton_physical_content_gr} for a full discussion, which we shall presently recapitulate in nuce. Modern-day relativists are inclined to find it natural to characterize space in intrinsic terms; i.e., to model it via a smooth manifold (possibly equipped with additional structures of topological, metrical etc. significance) and then to demand that all physically meaningful statements be invariant in the sense of being independent of any choice of coordinates. In the modern view, then, the physical objects of the model will be represented by equivalence classes of tensors or other potentially observable entities and functions of them under diffeomorphism. But Einstein means something quite different by his principle of general covariance. Norton phrases it thus: Einstein does not use differentiable manifolds at all but open sets in Euclidean space on which the coordinate systems given by the $x^{0,\ldots,3}$-axes are \textit{not} charts but correspond to a representation of an inertial frame in a physically possible space-time. For convenience, one may call these \textit{Einstein-Grossmann number manifolds} (or, to be more precise, the model of space-time one would obtain upon taking a number of such open sets and identifying points in common, should more than one be needed to cover all of space-time). In modern terms, one has thereby at one's disposal a convenient means of specifying a congruence of curves.

Norton distinguishes covariance from Leibniz equivalence: for him, a diffeomorphism produces another model of a possible space-time theory while the principle of Leibniz equivalence adds the stipulation that what one gets through this procedure is physically the \textit{same} space-time as the one with which one started. In the course of his deliberations on the problem of formulating a successor theory to the special theory of relativity that would include gravitation, Einstein began to understand that there must be intermediate stages between Lorentz and general covariance---which view would of course be nonsensical if the principle of general covariance were to be a mere mathematical convenience devoid of physical content. For starting with Lorentz covariance, the coordinate description of Minkowskian space funds one with an embarrassment of riches, more than one has to right to expect (Norton: `problem of too much structure'), in that level sets of the time coordinate can be read off immediately as timelike hypersurfaces embodying a global notion of simultaneity, the timelines passing through every point specified on such a surface of simultaneity by the three spatial coordinates will correspond to bodies at rest and go over under Lorentz transformation into bodies moving with constant velocity in another inertial frame etc. Hence, one wishes to relax the degree of physicality embedded into the mathematical coordinates by permitting progressively larger classes of coordinate transformations.

Einstein's interpretation of the \textit{physical} significance of such an operation is just that asserted by the principle of equivalence; viz., that bodies at rest with respect to a uniformly accelerated frame, for instance, are to be regarded as equivalent to bodies subject to a constant gravitational force in the original frame, and so forth. Further consequences flow from the same point of view (such as the discovery that we have reason to suppose the existence of a so-called gravitational redshift); in view of this fact which is so important in underscoring the potential of relativity theory to enhance our understanding of \textit{physics}, we want to follow Einstein's lead and suggest the following
\begin{postulate}
	A principle of general covariance in Einstein's sense will be supposed, now for generalized tensors.
\end{postulate}

\begin{postulate}
An inertial frame in flat space-time is considered to be defined, in principle, by laying down a grid of rigid rods and placing stationary clocks at regular intervals. The Cartesian coordinate frame is then obtained by local measurements of orthogonality between the connecting rods and by a time coordinate at each given position in space that can be read off from the stationary clock located there. At higher order, the Cartesian coordinates so defined will be presumed to satisfy the principle of plenitude in the sense that the extended Minkowskian metric is supposed to be given by equation (\ref{generalized_minkowski_metric}) above.
\end{postulate}

\begin{postulate}
	Bodies in free space move along future-directed timelike jet geodesics of the generalized Minkowski metric.
\end{postulate}

\begin{postulate}
	The proper setting for the general theory of relativity is a pseudo-Riemannian geometry of higher order on the manifold of space-time (presumed to be differentiable); freely falling bodies move along jet geodesics \textup{(}in the limit when they are sufficiently small so as not to have any sensible effect on the large bodies, nearby or distant, that, by Mach's principle, may be regarded as responsible for establishing the background geometry\textup{)}.
\end{postulate}
In view of considerations such as these, therefore, it seems better to have recourse to Thirring's approach to the formulation of the general relativistic field equations which circumvents the need to take the trace. See \cite{thirring_vol_2}, {\S}4.2. Fortunately, the elegant extension of Cartan's formalism to the case of non-zero higher infinitesimals sketched in the previous section enables us to propose a rapid route to a good candidate for the field equation, based upon Thirring's discussion of the first-order case. Formally, the expressions will look just about the same, understanding the Levi-Civita connection to be replaced by its jet version and the exterior derivative $d$ promoted to $\text{\th}$. To wit, Thirring suggests the following five candidates for a macroscopic $4$-volume form exhaust the range of reasonable choices for what the extended gravitational part of the integrand of the Einstein-Hilbert action could be:

\begin{itemize}
	\item[$(1)$] \qquad $*1 = \sqrt{ |\hat{g}|} e^{0123}$ 
	\item[$(2)$] \qquad $R_{\alpha\beta} \wedge e^{\alpha\beta}$ 
	\item[$(3)$] \qquad $R_{\alpha\beta} \wedge *e^{\alpha\beta}$ 
	\item[$(4)$] \qquad $R_{\alpha\beta} \wedge R^{\alpha\beta}$ 
	\item[$(5)$] \qquad $R_{\alpha\beta} \wedge *R^{\alpha\beta}$ 
\end{itemize}
The first choice corresponds to a cosmological-constant term which we are free to include or to omit as we see fit. In any case, it will not by itself yield the desired local dynamics. The second can be seen to vanish identically; in fact, it is tantamount to casting the second Bianchi identity into an algebraic form, as the following lemma shows:

The third is the winning choice in Einstein's version of the general theory of relativity and it stands to reason it ought to continue to be valid in post-Einsteinian theory, so as to assure an unbroken transition in the scaling limit. This Ansatz proves desirable moreover because, written in Cartan form, the field equations immediately imply a conservation law for matter on the right hand side. The fourth possibility for the sought-for action is nugatory because it reduces to an exact 4-form (cf. Thirring's Problem 4.2.25,6), mutatis mutandis, \cite{thirring_vol_2}). The fifth would yield a non-abelian gauge-type theory. In view of the prospective unification of forces to be investigated in Part III, we should like to retain this term if at all possible.

Thus, the problem becomes to find a reasonable combination of those of the above possibilities that are viable. By way of motivation for the choice to be announced in a moment, let us remark on the role of dark energy. If it be viewed phenomenologically as a macroscopic contribution to the dynamics of space-time left over from microscopic interactions at the quantum-field theoretic level, there would seem to be no necessity to restrict it to assume the form of a constant amount of energy per unit volume. After playing around with possible constructions by means of which to extend the Einstein-Hilbert action to higher-order jets, rather, the following view suggests itself as natural. It could be supposed that dark energy enters as a constant background contribution to the \textit{curvature} of space-time. That is, we may write the generalized Riemannian curvature tensor in the following form,
\begin{equation}\label{curvature_with_dark_energy}
	R^\alpha_\beta - \lambda^2 \Xi^\alpha_\beta = R^\alpha_\beta - \lambda^2 B^{|\alpha|+|\beta|-2} e^\alpha \wedge e_\beta,
\end{equation}
where we separate out a factor of $\lambda^2$ (having dimensions of inverse length squared) and, when needed, an as-yet unknown cosmological scale $B$ (having dimensions of length) so that, for 1-jets, the cosmological term will be a 2-form having the same units as the curvature itself, viz., inverse length squared. For higher jets, the additional dimensional factors proportional to powers of length have been incorporated into the factor involving the parameter $B$. The minus sign ensures that the cosmological constant in the field equation is positive.

In equation (\ref{curvature_with_dark_energy}), $R^\alpha_\beta$ indicates the curvature 2-form arising from all other sources than dark energy, as usual, but of course, here, one also allows the multi-indices to run over jets of all orders. It will be convenient to gather all components into a single entity, the total curvature $R - \lambda^2 \Xi$, where
\begin{align}
	R &:= R^\alpha_\beta \otimes e_\alpha \wedge e^\beta \nonumber \\
	\Xi &:= \sum_{\alpha < \beta} e^\alpha \wedge e_\beta \otimes e_\alpha \wedge e^\beta.
\end{align}
The cosmological term thus defined corresponds to a covariant object, what is the content of the following proposition.

\begin{proposition}
	If $e^\alpha$ and $f^\alpha$ are two orthonormal co-frames, the expression $\Xi$ assumes with respect to either has the same form.
\end{proposition}
\begin{proof}
	Suppose the two co-frames are related by a possibly position-dependent matrix $A$ as follows:
	\begin{equation}
		B^{|\alpha|-1} e^\alpha = A^\alpha_\beta B^{|\beta|-1} f^\beta,
	\end{equation}
	where $A \in SO(p,q)$ where $p$ counts the number of timelike directions and $q$ the number of spacelike directions among jets up to $r$-th order. Recall that the multi-index $\alpha$ is timelike resp. spacelike according to whether $(-1)^{\alpha_0}$ equals $-1$ resp. $1$. Now, with respect to the first co-frame basis $\Xi$ assumes the form
	\begin{equation}\label{cosmological_in_coframe_basis}
		\Xi^\alpha_\beta = \frac{1}{2} \begin{pmatrix}
			0 & 1 & 1 & \cdots & 1 \\ 
			-1 & 0 & 1 & \cdots & 1 \\
			\vdots & -1 & 0 & \ddots & \vdots \\
			\vdots & & \ddots & \ddots & 1 \\
			-1 & \cdots & \cdots & -1 & 0 \\
		\end{pmatrix}.
	\end{equation}
	Under the change of co-frame, $\Xi$ transforms in the same way as the Maxwellian field-strength tensor would, namely, $A \mapsto A^t \Xi A$. But $A$ itself can be obtained from a product of elements of the form
	\begin{equation}\label{spacelike_factor}
		R = \begin{pmatrix}
			1 & 0 & \cdots & \cdots & \cdots & \cdots & 0 \\
			0 & \ddots & & & & & \vdots \\
			\vdots & & \cos \theta & & - \sin \theta & & \vdots \\
			\vdots & &  & \ddots & & & \vdots \\
			\vdots & & \sin \theta & & \cos \theta & & \vdots \\
			\vdots & & & & & \ddots & \vdots \\
			0 & \cdots & \cdots & \cdots & \cdots & \cdots & 1 \\
		\end{pmatrix}
	\end{equation}
	resp.,
	\begin{equation}\label{timelike_factor}
		S = \begin{pmatrix}
			1 & 0 & \cdots & \cdots & \cdots & \cdots & 0 \\
			0 & \ddots & & & & & \vdots \\
			\vdots & & \cosh \phi & & \sinh \phi & & \vdots \\
			\vdots & &  & \ddots & & & \vdots \\
			\vdots & & \sinh \phi & & \cosh \phi & & \vdots \\
			\vdots & & & & & \ddots & \vdots \\
			0 & \cdots & \cdots & \cdots & \cdots & \cdots & 1 \\
		\end{pmatrix},
	\end{equation}
	where one takes the former if the two columns resp. rows having non-zero off-diagonal components have the same signature and the latter if they have opposite signature. Then,
	\begin{align}
		\Xi &= \sum_{\alpha<\beta} e^\alpha \wedge e_\beta \otimes e_\alpha \wedge e^\beta \nonumber \\
		&= \frac{1}{2} \sum_{\alpha\beta} e^\alpha \wedge e_\beta \otimes e_\alpha \wedge e^\beta \nonumber \\
		&= \frac{1}{2} \sum_{\alpha\beta} A^\alpha_\mu f^\mu \wedge A_\beta^\nu f_\nu \otimes A_\alpha^{\mu^\prime} f_{\mu^\prime} \wedge A^\beta_{\nu^\prime} f^{\nu^\prime} \nonumber \\
		&= \frac{1}{2} \sum_{\alpha\beta} \left( A^\alpha_\mu  A_\alpha^{\mu^\prime} \right) \left(  A_\beta^\nu  A^\beta_{\nu^\prime} \right)
		f^\mu \wedge f_\nu \otimes f_{\mu^\prime} \wedge f^{\nu^\prime}.
	\end{align}
	But
	\begin{equation}
		A^\alpha_\mu  A_\alpha^{\mu^\prime} = (A^t)^\mu_\alpha A^\alpha_{\mu^\prime} =
		(A^t \eta A)^\mu_{\mu^\prime} =\eta^\mu_{\mu^\prime} = \delta_{\mu\mu^\prime}
	\end{equation}
	and similarly for the other factor. Therefore,
	\begin{equation}
		\Xi = \frac{1}{2} \sum_{\mu\nu} f^\mu \wedge f_\nu \otimes f_\mu \wedge f^\nu = \sum_{\mu<\nu} f^\mu \wedge f_\nu \otimes f_\mu \wedge f^\nu;
	\end{equation}
	i.e., $\Xi$ can be defined in the same way with respect to any co-frame.
\end{proof}

There results from this suggestion a natural proposal for the extended Einstein-Hilbert action, viz., to produce a volume form by wedging the Riemannian curvature of space-time so expressed with itself (after taking the Hodge-*). Therefore, we arrive at
\begin{definition}[Cf, Thirring, \cite{thirring_vol_2}, 4.2.4]\label{einstein_hilbert_action}
	The Einstein-Hilbert action in the post-Einsteinian general theory of relativity may be written as $\int \mathscr{L}_{EH}$, where
	\begin{equation}
		\mathscr{L}_{EH} = \frac{1}{16 \pi \varkappa} \left( R_{\alpha\beta} - \lambda^2 \Xi_{\alpha\beta} \right) \wedge * 
		\left(  R^{\alpha\beta} - \lambda^2 \Xi^{\alpha\beta} \right).
	\end{equation}
\end{definition}

In order to obtain generalized field equations, one needs some means of representing the stress-energy of neutral gravitating matter to appear on the right-hand side. Here, we may take our cue once again from Thirring. But one wishes to prepare the ground for the eventual unification of gravitation with gauge forces.
In view of what has already been said, the following proposal suggests itself as natural: we construe matter of all kinds, including but not limited to the effective contribution arising from dark energy, as a local perturbation to the curvature of space-time---with this difference, that, as everyone knows, matter can be in motion (unlike dark energy). Therefore, its bulk form may be represented through a vector field or, by lowering the spatial index, a jet field (current), where now we tensor with two indices to obtain a flow of curvature 2-forms; thus, we write the matter current as
\begin{equation}
	J \in \mathscr{J}^{\infty *} \otimes \mathscr{J}^{\infty} \otimes \mathscr{J}^{\infty *}.
\end{equation} 
This idea proves to be most fruitful. For, as we shall demonstrate in detail in Part III, the interpretation naturally suggests itself that when the curvature 2-form multi-indices range over infinitesimals of higher than first order, these ought to be equivalent to a gauge field. For the purposes of the present section, however, we want to limit ourselves to neutral bulk matter only, which in general relativistic theory goes under the name of dust. Thus, dust ought to be expressible by means of a matter current in which all gauge multi-indices of higher than first order are suppressed. Hence, we will write for the neutral matter current, or dust (denoting the density by $\varrho$ and the spatial part of the current by $X$),
\begin{equation}\label{dust_current}
	J = - \frac{1}{2} \lambda^2 \varrho X \otimes \mathrm{id}_{\mathscr{J}^1},
\end{equation}
in which the gauge dependence is just the identity operator on the image of the canonical injection of 1-vector fields $\mathscr{J}^1 \subset \mathscr{J}^\infty$, and zero elsewhere (that is, one must project orthogonally onto $\mathscr{J}^1$). From here, it is a short step to a proposed phenomenological form of the dust term in the action which will lead upon variation to a suitable stress-energy tensor on the right hand side of the field equation:
\begin{equation}\label{dust_lagrangian}
	\mathscr{L}_\mathrm{dust} = \frac{1}{4} ~\mathrm{tr}~ X \otimes \mathrm{id} \wedge * J = - \frac{1}{8} \lambda^2 \varrho ~\mathrm{tr}~ X \otimes \mathrm{id} \wedge * X \otimes \mathrm{id} = - \frac{1}{4} \lambda^2 \varrho X \wedge * X.
\end{equation}
Here, the trace indicates a contraction with respect to the gauge indices of $X \otimes \mathrm{id}$ resp. $J \otimes \mathrm{id}$, taken as usual by raising one index with the metric tensor $g^{-1}$. Be careful in equation (\ref{dust_lagrangian}) to understand that the wedge product and the Hodge-* operate only on the spatial index of the dust current. Let us comment that the minus signs in equations (\ref{dust_current}) and (\ref{dust_lagrangian}) reflect the fact that the cosmological constant is actually positive; i.e., $\Lambda =  \frac{1}{2} \lambda^2 > 0$.

These interpretations of the form in which the Einstein-Hilbert and dust actions extend to higher order being in place, we are led to

\begin{postulate}
	The generalized field equations are given by theorem \ref{post_Einsteinian_field_equ} below.
\end{postulate}

\subsubsection{Variational Formulation}

The approach to be followed in this subsection will be simply to incorporate analogues of whatever may be needed from Thirring vol. 2, {\S}1.2, then to follow {\S}{\S}4.2.6, 4.2.9 in \cite{thirring_vol_2}.

\begin{lemma}\label{curv_variation}
	The curvature part of the variation of the Einstein-Hilbert action yields merely a total differential and therefore may be neglected; i.e., we have that $*e^{\alpha\beta} \wedge \delta R_{\alpha\beta} = \text{\th} \left( * e^{\alpha\beta} \wedge \delta \omega_{\alpha_\beta} \right)$.
\end{lemma}
\begin{proof}
	Here again, the substance of Thirring's solution applies where we merely must take care to ensure that the additional cross-terms appearing in the Leibniz rule for generalized derivations cancel out in the final result. We proceed thus:
	
	First, we argue that the jet affine connection corresponding to the Levi-Civita jet connection is represented by an antisymmetric matrix: $\omega^\alpha_\beta=-\omega^\beta_\alpha$. The reason for this may be seen as follows. With respect to the orthonormal frame or coframe bases, the coefficients of $g^{-1}$ resp. $g$ are given by the Kronecker delta symbol, which is to say, identically constant when $\text{\th} g_{\mu\nu} = \text{\th} g^{\mu\nu} = 0$. Write out the compatibility equation, for instance, $\nabla_X g^{-1} = 0$ in terms of the jet affine connection and one obtains a formula of the form (schematically),
	\begin{equation}\label{jet_affine}
		\omega + \omega \omega^t + \omega^t = 0,
	\end{equation}
	where as in lemma I.3.8 we renormalize to eliminate combinatorial factors. Let now the jet affine connection be broken down into its symmetric and antisymmetric parts: $\omega = A + B$ and substitute into Equation (\ref{jet_affine}) to obtain:
	\begin{equation}
		A+B+AA-AB-BA+BB+A-B=0.
	\end{equation}
	But by reason of symmetry, $AB=BA$; hence, $2A+AA-BB=0$. This result expresses the components of the symmetric part $A$ of given order in terms of something quadratic in the components of $A$ and $B$ strictly lower in order; thus, to start with the first-order components of $A$ must vanish, but then so too must the second-order components and so on inductively until we reach the conclusion that $A=0$, whence also $BB=0$ for the cross-terms of the antisymmetric part. Thus, we are left with the statement that $\omega+\omega^t=0$, as claimed.
	
	Second, we have the relation $\text{\th}_\mu *e^{\alpha\beta} = - 2 \omega^\alpha_\gamma(\partial_\mu) \wedge *e^{\gamma\beta}$. This corresponds to the formula 1.2.26 in Thirring with the exterior derivative $d$ replaced by $\text{\th}$. Fortunately, however, Thirring's solution to problem 1.2.39,8) goes through for each component $\text{\th}_\mu$ taken separately, then sum over multi-indices to get $\text{\th}$. Note: the derivation appeals to the relation 1.2.17 in Thirring, or the statement to the effect that $*e^{\alpha_1\cdots\alpha_p}$ transforms naturally like $e^{\alpha_1\cdots\alpha_p}$ under change of frame, so we must restrict to changes of frame that respect the block-diagonal decomposition (else our method of defining Hodge-* becomes inapplicable); we omit to reproduce Thirring's full derivation here as it is fairly technical and merely remark that it generalizes simply to the higher order case so that the analogous formula stands.
	
	Putting everything together, as in Thirring we may write,
	\begin{align}
		\text{\th} \left( *e^{\alpha\beta} \wedge \delta \omega_{\alpha\beta} \right) &=
		-2 \omega^\alpha_\gamma \wedge *e^{\gamma\beta} \wedge \delta \omega_{\alpha_\beta}
		+ *e^{\alpha\beta} \wedge \text{\th} \delta \omega_{\alpha\beta} + \cdots \nonumber \\
		&= *e^{\alpha\beta} \wedge \left( \text{\th} \delta \omega_{\alpha\beta} +
		2 \omega^\gamma_\beta \wedge \delta \omega_{\gamma\beta} \right) \nonumber \\
		&= *e^{\alpha\beta} \wedge \delta R_{\alpha\beta},
	\end{align}
	as was to be shown. Here, the cross-terms indicated by ellipsis on the first line cancel due to a relative minus sign in the vee expansions, cf. lemma I.4.6.
\end{proof}

\begin{theorem}[Field equations; cf. Thirring,  \cite{thirring_vol_2}, 4.2.9]\label{post_Einsteinian_field_equ}Field configurations with respect to which the Einstein-Hilbert action is stationary under infinitesimal variation satisfy the following field equations:
	\begin{align}\label{field_equ_in_explicit_form}
		- \frac{1}{2} \left( * e_{\alpha\beta\gamma} \right) \wedge \left( R^{\beta\gamma} - \lambda^2 \Xi^{\beta\gamma} \right) &= 8 \pi \varkappa * t_\alpha \\
		\text{\th} *J &= 0,
	\end{align}
	where the stress-energy tensor on the right-hand side is given by
	\begin{equation}\label{stress_energy_explicit_form}
		* t_\alpha = \frac{1}{2} \varrho \left( (i_\alpha X) *X + X \wedge i_\alpha *X \right).
	\end{equation}
\end{theorem}
\begin{proof}
	For the gravitational part, first expand the Einstein-Hilbert action as
	\begin{equation}
		R_{\alpha\beta} \wedge * R^{\alpha\beta} - 2 \lambda^2 B^{|\alpha|+|\beta|-2} R_{\alpha\beta} \wedge * e^{\alpha\beta} +
		\lambda^4 e_{\alpha\beta} \wedge * e^{\alpha\beta}.
	\end{equation}
	The most adept way to handle the variational problem would be to regard the two as independent of each other and to vary them separately, whereupon the constraint (I.4.108) may be imposed via a Lagrange undetermined multiplier. We intend to display how this may be done below in {\S}III.1.3.3. For the time being, the following consideration implies that, as far as gravitational physics is concerned, the $R_{\alpha\beta} \wedge * R^{\alpha\beta}$ term drops out of the Einstein-Hilbert action. For in view of lemma \ref{curv_variation}, we may write
	\begin{equation}
		* R^{\alpha\beta} \wedge \delta R_{\alpha\beta} = \text{\th} \left( * R^{\alpha\beta} \wedge \delta \omega_{\alpha\beta} \right) +
		\mathrm{terms~in~} R^{\alpha\beta}_{,\gamma}.
	\end{equation}
	Far below the Planck scale, the terms in derivatives of the curvature will be entirely negligible compared to terms retained. So we are left with two terms, the first being,
	\begin{equation}
		\delta \left( * e_{\alpha\beta} \wedge R^{\beta\gamma} \right) =
		\delta \left( * e_{\alpha\beta} \right) \wedge R^{\gamma\beta} +
		* e_{\alpha\beta\gamma} \wedge \delta R^{\beta\gamma}.
	\end{equation}
	In view of lemma \ref{curv_variation} we may ignore the second term. But by lemma I.4.11,
	\begin{equation}
		\delta \left( * e^{\alpha\beta} \right) = \delta e^\gamma \wedge i_\gamma * e^{\alpha\beta} = \delta e^\gamma \wedge * e^{\alpha\beta}_\gamma
	\end{equation} 
	(formally just as in Thirring). The second (cosmological) term in the action may be handled as follows. First, we have $e_{\alpha\beta} \wedge * e^{\alpha\beta} = *1$. Then, by the same argument leading to $\delta * e^\alpha = \delta e^\beta \wedge i_\beta e^\alpha = \delta e^\beta \wedge * e^\alpha_\beta$, we have also $\delta * 1 = \delta e^\beta \wedge * e_\beta$ (cf. Thirring, \cite{thirring_vol_2}, Problem 4.2.25,4). Upon dividing by $\lambda^2 B^{|\alpha|+|\beta|-2}$ and multiplying by $4 \pi \varkappa$, we are left with the variation of the Einstein-Hilbert action as
	\begin{equation}
		\delta e^\alpha \wedge \left( \frac{1}{2} ( *e_{\alpha\beta\gamma} ) \wedge R^{\beta\gamma} + \frac{1}{2} \lambda^2 * e_\alpha \right).
	\end{equation}
	
	This leaves the problem to evaluate the variation of the matter part, viz., $\delta (X \wedge *J)$. From lemma I.4.11 again,
	\begin{equation}
		e^\alpha \wedge *J = J \wedge * e^\alpha,
	\end{equation}
	so that
	\begin{equation}
		e^\alpha \wedge \delta *J = \delta J \wedge * e^\alpha + J \wedge \delta * e^\alpha - (\delta e^\alpha) \wedge *J.
	\end{equation}
	But by writing out $* e^\alpha$ in terms of a product of 1-form components and collecting the variations, we find just as Thirring does that
	\begin{equation}
		\delta * e^\alpha = \delta e^\beta \wedge i_\beta * e^\alpha.
	\end{equation}
	Substitute into the above, multiply by $X_\alpha$ and sum:
	\begin{equation}
		X \wedge \delta *J = \delta X \wedge *J - \delta e^\alpha \left[ X \wedge i_\alpha *J + (i_\alpha X) *J \right].
	\end{equation}
	Therefore,
	\begin{equation}
		\delta \left( - \frac{1}{2} X \wedge *J \right) = - \delta X \wedge *J +\frac{1}{2} \delta e^\alpha \left[ X \wedge i_\alpha *J + (i_\alpha X) *J \right].
	\end{equation}
	The quantity in square brackets becomes just what we have defined to be the stress-energy tensor $* t_\alpha$ (where as an advantage the derivation illuminates why it takes the indicated form). Every step copies the corresponding step in Thirring precisely; the sole difference being that the expressions are to be understood as running over multi-indices $1 \le |\alpha| \le r$, with now $r>1$. The sole difference arising from the gauge indices contained in equation (\ref{dust_current}) is that the trace yields a factor of $\eta_{\alpha\beta} \mathrm{id}^\beta_\gamma \mathrm{id}^\gamma_\alpha = 2$, where the summations need run only over jets of first order.
	
	Assembling the results found up to now produces for the variation of the Einstein-Hilbert action,
	\begin{equation}
		\delta \mathscr{L} = - \delta X \wedge *J + \delta e^\alpha \wedge \lambda^2 \left[ 
		2 * t_\alpha + \frac{1}{8 \pi \varkappa} * e_{\alpha\beta\gamma} \wedge R^{\beta\gamma} + \frac{1}{16 \pi \varkappa} \lambda^2 * e_\alpha \right]
	\end{equation}
	(up to a total differential). We almost have the expression in its desired final form. The variations $\delta J$ and $\delta e^\alpha$ are independent, of course, but it will be convenient to write the former in terms of a scalar function, i.e., with $J = \text{\th}S$ we have the total differential
	\begin{equation}
		\text{\th} \left( \delta S *J \right) = \text{\th} \delta S \wedge *J + \delta S \text{\th} *J + \cdots.
	\end{equation}
	The cross-terms can be seen to drop out by the argument that by now will be familiar: for $\delta S$ being scalar while $*J$ is a 3-form, there will be a relative minus sign in the vee expansion of $\text{\th} (\delta S *J)$ versus that of $\text{\th} (*J \delta S)$ leading to cancellation. From what has been said, we end up with the variation of the Einstein-Hilbert action in the following form (modulo a total differential again):
	\begin{equation}
		\delta \mathscr{L} = - \delta S \text{\th} *J + \delta e^\alpha \wedge \lambda^2 \left[ 
		2 * t_\alpha + \frac{1}{8 \pi \varkappa} * e_{\alpha\beta\gamma} \wedge R^{\beta\gamma} + \frac{1}{16 \pi \varkappa} \lambda^2 * e_\alpha \right],
	\end{equation}
	from which we may immediately read off the Euler-Langrange equations as indicated in the statement of the theorem.
\end{proof}
The version figuring in Cartan's formalism into which the field equation has been cast in theorem \ref{post_Einsteinian_field_equ} mirrors Thirring's (here we have left out the electromagnetic part in his formulation, by design, as we intend in Part III to show how it appears in our theory automatically, without having to be put in by hand as Thirring does). Let us briefly review how to get to the usual formulation from this. Rewrite the left hand side using the identities in lemma I.4.11, then perform the appropriate contractions on either side. The field equation may be expressed most conspicuously along the lines Thirring sketches in his {\S}4.2.11. First note that by the lemma I.4.11 we may write the left-hand side as,
\begin{equation}
	-\frac{1}{2} * e_{\alpha\beta\gamma} \wedge R^{\beta\gamma} = \left( *e_\alpha i_\beta i_\gamma + *e_\beta i_\gamma i_\alpha + *e_\gamma i_\alpha i_\beta \right) R^{\beta\gamma}, 
\end{equation}
whence we obtain after reducing the right-hand side,
\begin{equation}
	* R_\alpha - \frac{1}{2} *e_\alpha R^\beta_\beta,
\end{equation}
lastly apply the Hodge-* to both sides of the field equations,
\begin{equation}
	R_\alpha - \frac{1}{2} e_\alpha R + \frac{1}{2} \lambda^2 e_\alpha = 8 \pi \varkappa t_\alpha.
\end{equation}
This expression reduces to something parallel to the form in which Einstein originally writes it when we substitute $R_\alpha = R_{\alpha\beta} e^\beta$, $e_\alpha = g_{\alpha\beta}e^\beta$ and $t_\alpha = T_{\alpha\beta}e^\beta$:
\begin{equation}\label{field_equ_conventional_form}
	R_{\alpha\beta} - \frac{1}{2} g_{\alpha\beta} R + \frac{1}{2} \lambda^2 g_{\alpha\beta} = 8 \pi \varkappa T_{\alpha\beta}.
\end{equation}
Two things must be borne in mind in this context: first, we recover the usual \textit{form} of the 1-jet field equation of the general theory of relativity (viz., the case $|\alpha|=|\beta|=1$), where it must now be understood that the Ricci tensor and scalar curvature, viz., $R_{\alpha\beta} = R^\gamma_{\alpha\gamma\beta}$ and $R=R^{\gamma\delta}_{\delta\gamma}$ incorporate infinitesimals of higher than first order arising from the components in the contractions with $|\gamma|, |\delta|>1$. Second, when we allow $\alpha$ and $\beta$ to run over multi-indices of all orders, we see that equation (\ref{field_equ_conventional_form}) renders explicit the presence of the entire hierarchy of higher-order sectors.

\subsection{Weak-Field Limit}\label{weak_field_limit}

The weak-field limit holds to a good approximation within the confines of the solar system. A dimensionless measure of the strength of the gravitational field is given by the quantity $GM/c^2r$, where $G = 6.67428 \times 10^{-11} \, \text{m}^3\text{kg}^{-1}\text{s}^{-2}$ is Newton's gravitational constant, $c = 2.99792458 \times 10^5 \, \text{km}/\text{s}$ is the speed of light in vacuum, $M$ is the mass of a typical body and $r$ its radius. For the sun, $M_\odot = 1.9891 \times 10^{30} \, \text{kg}$, $R_\odot = 6.955 \times 10^5 \, \text{km}$ and $GM_\odot/c^2R_\odot = 2.124 \times 10^{-6}$, while for the earth, $M_\oplus = 5.9736 \times 10^{24} \, \text{kg}$, $R_\oplus = 6,371 \, \text{km}$ and $GM_\oplus/c^2R_\oplus = 6.963 \times 10^{-10}$.

We do not intend thoroughly to investigate here the conditions under which one recovers Einstein's general theory of relativity and hence Newtonian mechanics as well in the non-relativistic limit; it will suffice, as a first stab, to recover a simple wave equation and solution in integral form for a weak gravitational field, without showing whether the contributions dropped are in fact negligible (which would call for a difficult analysis beyond the scope of the present work). Note: a weak field is a relative notion; the order of the jets matters too. We shall retain the jets and higher tangents to second order in infinitesimals. It constitutes a further key assumption that it is indeed valid merely to truncate at second order in this way.

\subsubsection{Solution to the Field Equations in the Weak-Field Limit}

Within the prescribed limits, our analysis proves to be fairly straightforward if we follow Thirring's lead (\cite{thirring_vol_2}, {\S}4.2.15). Let $x^{0,\ldots,3}$ denote Cartesian coordinates in Minkowskian space and write,
\begin{equation}
	e^\alpha = dx^\alpha + \varphi^\alpha_\beta e^\beta, \qquad \varphi_{\alpha\beta} = \eta_{\alpha\gamma}\varphi^\gamma_\beta = \varphi_{\beta\alpha},
\end{equation}
for the orthonormal coframe basis (here, of course, the multi-index runs over $1 \le |\alpha| \le 2)$. Aside: the presumed symmetry of $\varphi$ will be justified by the result of the calculation. Suppose $|\varphi^\alpha_\beta| \ll 1$ uniformly across space and time (along with all its derivatives). Then we have
\begin{equation}
	\text{\th} e^\alpha = \varphi^\alpha_{\beta,\gamma} dx^\gamma \wedge dx^\beta
\end{equation}
and it follows to first order in small quantities that
\begin{equation}\label{jet_affine_weak_limit}
	\omega^\alpha_\beta = \left( \varphi^\alpha_{\gamma,\beta} - \varphi_{\beta\gamma,}^\alpha \right) dx^\gamma
\end{equation}
(as can be seen by substitution into the definition of the jet affine connections to this order). In this case,
\begin{align}
	\text{\th} \omega^\alpha_\beta &= \left( \varphi^\alpha_{\gamma,\beta\mu} - \varphi_{\beta\gamma,\mu}^\alpha \right) dx^\mu \wedge dx^\gamma \\
	i_\alpha \text{\th} \omega^\alpha_\beta &= \left( \varphi^\alpha_{\gamma,\beta\alpha} - \varphi_{\beta\gamma,\alpha}^\alpha - \varphi^\alpha_{\alpha,\beta\gamma} + \varphi_{\beta\alpha,\gamma}^\alpha\right) dx^\gamma.
\end{align}
Any derivation of the weak-field limit in the general theory of relativity will be facilitated by suitable choice of coordinates. Thirring's approach, which we are following methodically here, renders it particularly clear why the so-called harmonic condition turns out to be convenient.

\begin{proposition}\label{existence_harmonic_coord}
	There exist harmonic coordinates in which $\text{\th} * \text{\th} x^\alpha = 0$, $\alpha=0,\ldots,3$.
\end{proposition}
\begin{proof}
	Start from arbitrary coordinates in the vicinity of a given point and denote the coframe basis by $\xi^{0,\ldots,3}$. From the fact that $\text{\th} * \xi^0$ is closed, it follows from the Poincar{\'e} lemma I.4.4 that (locally, of course) it is also exact; in other words, there exists a differential form $\lambda$ such that
	\begin{equation}
		\text{\th} * \xi^0 = \text{\th} \lambda.
	\end{equation}
	Since we are working locally, up to a constant which we are free to set to zero we then have that
	\begin{equation}
		* \xi^0 = \lambda \qquad \mathrm{or} \qquad \xi^0 = \pm * \lambda,
	\end{equation}
	a closed 1-form. Hence, by the Poincar{\'e} lemma again locally there must be a smooth function $y^0$ (0-form) such that $\xi^0 = \text{\th} y^0$.
	
	We wish to do the same for $\xi_1$. Since it is linearly independent of $\xi_0$, we can have $f \xi^0 + g \xi^1 = 0$ identically in a neighborhood only if $f=g=0$ identically. Therefore, the smooth function $y^1$ found via the above procedure has to be such that $\text{\th} y^1$ is linearly independent of $\text{\th} y^0$ or, that is to say, will serve as a second coordinate function. In the same way, we get a third and fourth coordinate function $y^2$ and $y^3$. By construction, however, the coordinates $y^{0,\ldots,3}$ satisfy the harmonic condition.
\end{proof}

\begin{lemma}\label{harmonic_lemma}
	In harmonic coordinates we have the following relation:
	\begin{equation}
		\varphi_{\beta\alpha,}^\alpha = \frac{1}{2} \varphi^\alpha_{\alpha,\beta}
	\end{equation}
	(in the weak-field limit only, needless to say).
\end{lemma}
\begin{proof}
	Thirring's derivation applies mutatis mutandis: if $\text{\th} x^\alpha = e^\alpha - \varphi^\alpha_\beta e^\beta$ (as will be the case to first order in small quantities), we may compute:
	\begin{align}
		0 &= \text{\th} * \text{\th} x^\alpha = \text{\th} \left( * e^\alpha - \varphi^\alpha_\beta * e^\beta \right) \nonumber \\
		&= - \omega^\alpha_\beta \wedge * e^\beta - \varphi^\alpha_{\beta,\gamma} e^\gamma \wedge * e^\beta + \cdots \nonumber \\
		&= \left( \varphi_{\beta\gamma,}^\alpha - \varphi^\alpha_{\gamma,\beta} - \varphi^\alpha_{\beta,\gamma} \right) e^\gamma \wedge * e^\beta.
	\end{align}
	Here, the cross-terms drop out because by the hypothesis $\text{\th} x^\alpha = e^\alpha - \varphi^\alpha_\beta e^\beta$ they will enter only to second order in small quantities. 	
	By lemma I.4.11, however, $e^\gamma \wedge * e^\beta = \eta^{\gamma\beta} * 1$. The indicated relation follows immediately.	
\end{proof}

It remains now merely to substitute into the field equations. Remembering that we are accounting for multi-indices up to the second order, it will be illuminating to write the field equations in a form that brings out the first- and second-order parts explicitly. Also, as usual, it will be convenient to relate the scalar curvature to the stress-energy tensor by taking the trace on both sides. We have the option, however, first to take a partial trace over the first-order resp. second-order components. Therefore, let us adopt the following notation: $R^\mu_\nu$ will now denote the first-order part of the Ricci curvature tensor while $P^{\mu\nu}_{\lambda\sigma}$ will denote its second-order part and $Q^\mu_{\lambda\sigma}$ the coupled first- and second-order part. Here, the Greek indices range over $0,\ldots,3$ only. Then, accordingly, we write $R = R^\mu_\mu$ and $P=P^{\mu\nu}_{\mu\nu}$ for the partial traces. Likewise for the stress-energy tensor defined by $t_\alpha = T_{\alpha\beta} e^\beta$ and its partial traces $t = T^\mu_\mu$ and $s = T^{\mu\nu}_{\mu\nu}$. We have arrived at a system of equations involving $R_\alpha$, $P_{\alpha\beta}$ and $Q_{\alpha,\beta\gamma}$ together with $R$ and $P$ on the left-hand side and $t_\alpha$, $t$ and $s$ on the right-hand side. When taking the partial traces it is well to recall that at second and higher orders we must enumerate the basis elements appropriately, i.e., supply them with the right combinatorial factors. Here, solely at second order, it means that off-diagonal basis elements  have to be counted twice while those on the diagonal, only once each. Therefore, the partial trace over the 1-jet sector yields a factor of 4 while over the 2-jet sector, a factor of 16. Putting everything together yields, after a little simple algebra, the following system of equations:
\begin{align}\label{gtr_field_equ}
	R_\mu &= 8 \pi \varkappa \left( t_\mu - \frac{1}{2} e_\mu \left( t + \sigma \right) \right) \\
	P_{\mu\nu} &= 8 \pi \varkappa \left( s_{\mu\nu} - \frac{1}{2} e_{\mu\nu} \left( t + \sigma \right) \right) \\
	Q_{\mu,\nu\lambda} &= 8 \pi \varkappa u_{\mu,\nu\lambda}.
\end{align}
Here, we have chosen to isolate the second-order correction in terms of the quantity $\sigma := P/8 \pi \varkappa$. The reason behind this choice will become apparent in a moment. If we just evaluate the trace of the second line we obtain an expression for $\sigma$:
\begin{equation}
	\sigma = s - 8t - 8\sigma;
\end{equation}
hence, $\sigma = \dfrac{1}{9}(s - 8t)$. As we see, the field equations assume a simple form in these terms. The quantity $\sigma$ represents the correction at second order to Einstein's field equation, i.e., the first line of equation (\ref{gtr_field_equ}). Somehow, the partial trace of $s_{\mu\nu}$ ought to contain $8t$ as its leading contribution. The reason why becomes evident when we consider the phenomenological Ansatz for the stress-energy tensor at second order, which in analogy to the first-order case we write as $T = X \odot P$. Here, $X$ is a higher-order tangent vector that represents the flow of matter and $P$ is proportional to the density of matter, i.e., mass-energy per unit volume, denoted $\varrho$. Its first-order components are to be interpreted in the usual manner as generating a 1-parameter translation group in space-time. The simplest postulate upon which we settle above would be to expect that, in a frame in which matter is locally at rest, the $X = \partial_0 + \partial_{00}$. This postulate requires modification. Recall that by our conventions $\partial_{00} = \dfrac{1}{2} \dfrac{\partial^2}{\partial t^2}$. To compensate for the $\frac{1}{2}$ factor, we should suppose for the velocity field of matter at rest $X = \partial_0 + 2 \partial_{00}$ and presume its momentum to be given just by $P= \varrho X$. At this point, a complication enters: the Ansatz (\ref{generalized_minkowski_metric}) for the generalized Minkowski metric implies that with respect to the Cartesian coordinates defining an inertial frame the second-order part picks up a coefficient of two; i.e., $d^{00} \otimes d^{00}$ enters as $2 d^{00}\otimes d^{00}$ etc. Thus, when we lower the index on the 2-velocity we find
\begin{equation}
	\hat{\eta} X = - d^0 + 4 d^{00}.
\end{equation}
For the stress-energy we then get in the original Cartesian frame,
\begin{equation}
	T_\mu^\nu = \varrho \hat{\eta} X \otimes X = \varrho d^0 \otimes \partial_0 - 2 \varrho d^0 \otimes \partial_{00} - 4 \varrho d^{00} \otimes \partial_0 + 8 \varrho d^{00} \otimes \partial_{00}.
\end{equation}
Taking the partial traces leads to $t = \varrho$ and $s = 8 \varrho$. Hence, the dominant part of $s$ is just $8t$, from which it follows that $\sigma = s - 8t = 0$. When the dust is not strictly at rest but consists of a gas of particles in relative motion to one another, of course, the stress-energy tensor acquires pressure terms and $\sigma$ may become small and non-zero. 

In any case, the ground has been prepared to derive the field equations in the weak-field limit. In Cartan's second structural formula for the Ricci curvature, the terms in $\omega \wedge \omega$ will drop out to the order of approximation we are considering. Then we find from equation (\ref{jet_affine_weak_limit}),
\begin{equation}
	R^{\alpha\beta} = \left( \varphi^{\alpha\beta}_{\gamma,\delta} - \varphi^{\beta\alpha}_{\gamma,\delta} \right) d^\delta \wedge d^\gamma.
\end{equation}
The field equations follow as
\begin{equation}
	- \varphi^\alpha_{\beta\gamma,\alpha} = 8 \pi \varkappa \left(
	T_{\beta\gamma} - \frac{1}{2} \eta_{\beta\gamma} (T+S) \right).
\end{equation}
In terms of the d'Alembertian to second order, viz.,
\begin{equation}
	\Box := - \partial_0^2 + \partial_1^2 + \partial_2^2 + \partial_3^2 + 
	\partial_{00}^2 - \partial_{01}^2 - \partial_{02}^2 - \partial_{03}^2 +
	\partial_{11}^2 + \partial_{22}^2 + \partial_{33}^2 + \partial_{12}^2 + \partial_{13}^2 + \partial_{23}^2,
\end{equation}	
we have
\begin{equation}\label{field_equ_weak_limit}
	\Box \varphi_{\beta\gamma} = 8 \pi \varkappa \left( T_{\beta\gamma} - \frac{1}{2} \eta_{\beta\gamma} (T+S) \right),
\end{equation}	
to the indicated degree of approximation. It is well to recall, at this juncture, that the curvature 2-form is built out of multilinear combinations of jets and a jet is just a formalization of a differential. The condition for a series expansion into jets of differing order to make sense in the first place is that the typical differential must have a magnitude on the order of $(B/b)^{|\alpha|}$, when expressed in dimensionless terms, where $B$ is the spatial scale of the gravitational system under consideration and $|\alpha|$ is the order of the jet. In consequence, the second-order part of the curvature 2-form $P$ will be down by two powers of $B/b$ relative to its first-order part $R$. From equation (\ref{field_equ_weak_limit}) one reads off then that the partial trace $S$ likewise will be down by two powers of $B/b$ unless the configuration of matter happens somehow to be singular, causing the series expansion to break down. But such an eventuality would be atypical. Therefore, $|S| \ll |T|$ and we may disregard it in the line defining the first-order part of the jet affine connection in the weak-field limit. For the same reason, the second-order terms in the d'Alembertian may be expected to be negligible by two powers of $B/b$ compared to those of first order.

The weak-field equations can be solved under the indicated conditions, as is usual in electrodynamics, by means of a retarded Green's function (cf. \cite{thirring_vol_2}, 4.2.18):
\begin{equation}\label{Green_fctn_solution}
	\varphi_{\alpha\beta}(\bar{x}) = \varphi^0_{\alpha\beta} + 8 \pi \varkappa \int d^4 x D^{\mathrm{ret}}(\bar{x}-x)\left( T_{\alpha\beta}(x)-\frac{1}{2}\eta_{\alpha\beta}T^\gamma_\gamma(x) \right)
\end{equation}
with the boundary conditions
\begin{equation}
	\Box \varphi^0_{\alpha\beta} = 0, \qquad \varphi^{0\alpha}_{\beta\alpha,}
	= \varphi^{0\alpha}_{\alpha,\beta}.
\end{equation}
For the retarded Green's function, one may take Thirring's formulae in terms of the Dirac delta distribution or the Heaviside step function (\cite{thirring_vol_2}, 2.2.7):
\begin{equation}
	D^{\mathrm{ret}}(x) = \frac{\delta(r-t)}{4 \pi |r|} = \frac{\delta(x^2)}{2\pi} \Theta(t), \qquad r = |\mathbf{x}|, \qquad x^2 = - x_0^2 + \mathbf{x}^2.
\end{equation}
The essential point of equation (\ref{field_equ_weak_limit}) (taking into account the subsequent discussion) is that the components of various order decouple from one another, so that the first-order components of the jet affine connection depend on the first-order components of the stress-energy, and likewise for second order and mixed first and second order.

\begin{corollary}
	The post-Einsteinian general theory of relativity reduces to the Einsteinian case in the scaling limit when the distances and time intervals considered are small compared to the fundamental unit of length. Ergo, it also yields Newton's law of universal gravitation in the non-relativistic limit.
\end{corollary}

\subsubsection{Recovery of the Relativistic Equation of Motion for Orbiting Bodies in the Solar System}

As an instructive illustration of how additional terms enter into equation (\ref{field_equ_weak_limit}) without disrupting the Einsteinian limit, we derive again the relativistic equation of motion for the orbits of the planets in the solar system. The key observation is that we may put $T = X \odot P$ and resolve the right hand side of equation (\ref{field_equ_weak_limit}) into independent contributions in the weak-field limit. From equation (\ref{Green_fctn_solution}), the functional form of the perturbative solution does not depend on the lower indices. We know already that at 1-jet order, a stress-energy tensor corresponding to a neutral massive body at rest, viz., $T = d^0 \odot d^0$ (lowering indices to bring the stress-energy into the form assumed by equation (\ref{Green_fctn_solution})) yields $g_{0,0} = - 1 + 2 GM_\odot/r$. Thus, in the 2-jet case the additional terms will look the same, up to possibly sign and a numerical prefactor.

These we now derive. First, consider the 2-jet contribution from $T = - d^0 \otimes 2 d^{00} - 2 d^{00} \otimes d^0 = - 4 d^0 \odot d^{00}$. Hence, it generates a net contribution to $g_{0,00}$ of minus four times $2GM_\oplus/r$. The same argument goes for the contributions from $d^0 \odot d^{11}$, $d^0 \odot d^{22}$ and $d^0 \odot d^{33}$. Go to the 22-sector. The $T = 2 d^{00} \odot 2 d^{00}$ will generate a contribution of four times $2GM_\oplus/r$ to $g_{00,00}$ just as $d^0 \odot d^0$ produces $2GM_\oplus/r$ in $g_{0,0}$. The same applies to all quadratic combinations of 2-jets, viz., $d^{00} \odot d^{11}, \ldots, d^{33} \odot d^{33}$. Then the right hand side of the jet geodesic equation including all 2-jet contributions should read,
\begin{equation}
	- \Gamma^\mu_{0,0} X^0 X^0 - 2 \Gamma^\mu_{0,00} X^0 X^{00} - \Gamma^\mu_{00,00} X^{00} X^{00} =
	- \Gamma^\mu_{0,0} - 2 \cdot (-4) \Gamma^\mu_{0,0} 2 - 4 \Gamma^\mu_{0,0} 2 \cdot 2 = - \Gamma^\mu_{0,0}. 
\end{equation}
Hence, the end result is that in the non-relativistic limit the contributions from the 12-sector and the 22-sector cancel, leaving us with the Newtonian law of universal gravitation as usual.

On physical grounds, something like this ought to occur in the relativistic case as well, at least for velocities typical of those of orbiting bodies in the solar system. We proceed now to show how. It will be convenient to refer to Wald's derivation of the geodesic equation of motion for the Schwarzschild metric in Einsteinian general theory of relativity \cite{wald}. Wald's streamlined approach relies on the existence of Killing vector fields, which by definition generate a 1-parameter family of diffeomorphisms under which the Riemannian metric must be invariant. In the usual calculus on manifolds, this statement is equivalent to the vanishing of the Lie derivative of the metric with respect to the Killing vector field. Here, we have not as yet formulated a sensible concept of Lie derivative for vector fields of higher than first order. Nevertheless, we can abstract the practical result as far as Wald is concerned. Therefore, we make the following
\begin{definition}
	Let $\xi$ be a vector field on the manifold $M$ equipped with the Riemannian metric $g$. Then we shall say that $\xi$ is a Killing field if we have that
	\begin{equation}
		\mathrm{Sym}~ \nabla g \xi = 0.
	\end{equation}
\end{definition}
Now the key property of Killing fields is that they embody a conservation principle for motion along geodesics, as is captured by the following proposition:
\begin{proposition}[Cf. Wald, \cite{wald}, Appendix C.3]\label{killing_invariant}
	Let $\xi$ be a Killing field for the Riemannian metric $g$ and let $X$ be a solution of the jet geodesic equation in the sense that $\nabla_X X=0$. Then $\langle \xi, X \rangle$ is a conserved quantity under the flow generated by $X$.
\end{proposition}
\begin{proof}
	We compute
	\begin{equation}
		\nabla_X \langle \xi, X \rangle = X^\alpha \frac{\alpha!}{\alpha_1!\alpha_2!} \langle \nabla_{\alpha_1} \xi, \nabla_{\alpha_2} X \rangle = 
		X^\alpha \frac{\alpha!}{\alpha_1!\alpha_2!} \nabla g \xi(\partial_{\alpha_1},\nabla_{\alpha_2} X).
	\end{equation}
	The first and last terms vanish because by hypothesis $\nabla_X X=0$ and $\nabla g \xi(X,X)=0$. Turn to the remaining terms. Note that the expression is symmetric under interchange of $\alpha_1$ and $\alpha_2$ and furthermore under interchange of the order of $\xi$ and $X$, due to the symmetry of the inner product. Consider each term $\nabla g \xi(\partial_{\alpha_1},\nabla_{\alpha_2} X)$ separately. If we symmetrize with respect to interchange of the two arguments, we will get zero by hypothesis. On the other hand, if we antisymmetrize and then add the same expression to itself with $\alpha_1$ and $\alpha_2$ interchanged, we would have
	\begin{equation}
		\nabla g \xi(\partial_{\alpha_1},\nabla_{\alpha_2}X) - \nabla g \xi(\nabla_{\alpha_2}X,\partial_{\alpha_1}) + 
		\nabla g \xi(\partial_{\alpha_2},\nabla_{\alpha_1}X) - \nabla g \xi(\nabla_{\alpha_1}X,\partial_{\alpha_2}).
	\end{equation}
	In virtue of the identity
	\begin{equation}
		\nabla g \xi \left( \partial_{\alpha_1} + \partial_{\alpha_2} + \nabla_{\alpha_1}X + \nabla_{\alpha_2}X, \partial_{\alpha_1} + \partial_{\alpha_2} + \nabla_{\alpha_1}X + \nabla_{\alpha_2}X \right) = 0
	\end{equation}
	we may rewrite as follows:
	\begin{equation}
		2 \nabla g \xi ( \partial_{\alpha_1}, \nabla_{\alpha_2} X) + 2 \nabla g \xi (\partial_{\alpha_2},\partial_{\alpha_1}X).
	\end{equation}
	If we follow the same steps starting from $\langle X, \xi \rangle$ we get instead
	\begin{equation}
		2 \nabla g \xi (\nabla_{\alpha_1} X,\partial_{\alpha_2}) + 2 \nabla g \xi (\partial_{\alpha_2}X,\partial_{\alpha_1}).
	\end{equation}
	But summing these two expressions leads to zero in virtue of the hypothesis that $\mathrm{Sym}~ \nabla g \xi = 0$. This concludes the proof.
\end{proof}

Killing fields are easy to find in space-times with the requisite symmetry, if the metric can be expressed in a form where it is independent of a coordinate. Here, the Schwarzschild metric is clearly stationary in the sense that its components do not depend on time and azimuthally symmetric in that the components are also independent of the azimuthal angle $\phi$.

\begin{proposition}
	Under the Schwarzschild metric in the weak-field limit, $\xi_0 = \partial_0 + \partial_{00}$ and $\xi_\phi = \partial_\phi + \partial_{\phi\phi}$ will be Killing fields.
\end{proposition}
\begin{proof}
	Since the components of $\xi_0$ and $\xi_\phi$ are constant in the $t,r,\theta,\phi$ coordinates, the claim amounts to the statements that $\Gamma^\mu_{\nu,0+00} = - \Gamma^\nu_{\mu,0+00}$ and likewise $\Gamma^\mu_{\nu,\phi+\phi\phi} = - \Gamma^\nu_{\mu,\phi+\phi\phi}$. In the weak-field limit, the formula (I.4.13) for the Levi-Civita connection reduces to 
	\begin{equation}
		\Gamma_{\mu\nu\lambda} = \frac{1}{2} \left(  
		\partial_\nu g_{\mu\lambda} - \partial_\mu g_{\nu\lambda} + \partial_\lambda g_{\mu\nu}
		\right).
	\end{equation}
	The first two terms display the required antisymmetry, while the third vanishes in this case since the components of the metric do not depend on either $t$ or $\phi$. Now, in the weak-field limit we may raise the first index using the Minkowskian metric, which is manifestly symmetric. Therefore, the claimed result holds.	
\end{proof}

By proposition \ref{killing_invariant}, we have two conserved quantities, viz., energy per unit rest mass and angular momentum per unit rest mass:
\begin{align}
	E &= - \langle \xi_0, X \rangle \\
	L &= \langle \xi_\phi, X \rangle.
\end{align}
Here, $X$ is given by
\begin{equation}
	X = \gamma \partial_0 +  \gamma \mathrm{v} \partial_{\mathrm{x}} + A,
\end{equation}
with $A$ given by the Lorentz boost of $2 \partial_{00}$, or
\begin{equation}
	A = 2 \gamma^2 \left( \partial_{00} + 2 \mathrm{v} \partial_{0\mathrm{x}} + \mathrm{vv} \partial_{\mathrm{xx}} \right). 	
\end{equation}
We now show that a cancellation occurs so that $E$ and $L$ are given by their usual expressions, up to a correction of order $\dfrac{GM_\odot}{c^2 a}\dfrac{\varv^2}{c^2} = \left( \dfrac{G M_\odot}{c^2 a} \right)^2$, where for Mercury $a = 0.3871 ~\mathrm{AU}$ and $\left( \dfrac{G M_\odot}{c^2 a} \right)^2 = 6.506 \times 10^{-16}$, i.e., negligibly small compared to everything else over timescales in the centuries. Since for Mercury's orbital speeds around $47.9 ~\mathrm{km/s}$, $\varv/c = 1.60 \times 10^{-4}$ and $(\varv/c)^2 = 2.55 \times 10^{-8}$, we may set $\gamma = 1$ with sufficient accuracy when it appears in a numerical coefficient of $\dfrac{2GM_\oplus}{r}$. Then evaluate the energy per unit rest mass to be
\begin{align}
	E &= - \langle \xi_0, X \rangle = - \langle \partial_0 + \partial_{00}, \gamma \partial_0 + \gamma \mathrm{v} \partial_{\mathrm{x}} + A \rangle \nonumber \\
	&= \left( 1 - \frac{2GM_\oplus}{r} \right) \gamma - \langle \partial_0, A \rangle - \langle \partial_{00}, X \rangle.
\end{align}
But to determine the second and third terms on the right, we can perform a change of coordinate to the locally comoving frame employ the value of the resulting expressions in free space, which are $\langle \partial_0, A \rangle = 0$ and $\langle \partial_{00}, 2 \partial_{00} \rangle = 2$, respectively. Hence,
\begin{equation}
	E = \left( 1 - \frac{2GM_\oplus}{r} \right) \gamma - 2,
\end{equation}
just as in Wald (\cite{wald}, Equation 6.3.12, apart from the shift by the additive constant). The result holds up to the terms neglected when going to the comoving frame, which are of order $\left( \dfrac{G M_\odot}{c^2 a} \right)^2$. Similarly, the angular momentum per unit rest mass will equal $r^2 \dot{\phi}$ up to small corrections in the non-relativistic limit. Insert into the relation,
\begin{align}
	3  &= \langle X,X \rangle = \langle \gamma \partial_0 + \gamma \mathrm{v} \partial_{\mathrm{x}}, \gamma \partial_0 + \gamma \mathrm{v} \partial_{\mathrm{x}} \rangle + \langle \gamma \partial_0 + \gamma \mathrm{v} \partial_{\mathrm{x}}, A \rangle + \langle A, \gamma \partial_0 + \gamma \mathrm{v} \partial_{\mathrm{x}} \rangle + \langle A,A \rangle \nonumber \\
	&= - \left( 1 - \frac{2GM_\odot}{r} \right) \gamma^2 + \left(1 - \frac{2GM_\odot}{r} \right)^{-1} \dot{r}^2 + r^2 \dot{\phi}^2 + 2 \gamma \langle \partial_0 + \mathrm{v} \partial_{\mathrm{x}}, A \rangle + \langle A, A \rangle \nonumber \\
	&= - \left( 1 - \frac{2GM_\odot}{r} \right) \gamma^2 + \left(1 - \frac{2GM_\odot}{r} \right)^{-1} \dot{r}^2 + r^2 \dot{\phi}^2 + 4.
\end{align}
In terms of the two constants of motion we get after collecting terms,
\begin{equation}
	\frac{1}{2}\dot{r}^2 + \frac{1}{2}\left( 1 - \frac{2GM_\oplus}{r} \right) \left(\frac{L^2}{r^2}+1 \right) = \frac{1}{2}(E+1)^2,
\end{equation}
i.e. precisely Wald's Equation 6.3.14 (modolo a shift of the energy per unit rest mass by an inessential constant). Therefore, disregarding the neglected corrections, we find again the same geodesic equation for the motion of the planets as in Einsteinian general theory of relativity.

\subsection{Schwarzschild Solution}
\label{schwarzschild_solution}

After the weak-field limit, the next simplest kind of solution to the generalized post-Einsteinian field equations to seek would be those having spherical symmetry in empty space \cite{schwarzschild}. Here, we treat only the static geometry of space-time itself and do not, until {\S}\ref{chapter_7}, investigate the present theory's implications for the motion of small bodies in the background geometry via the jet geodesic equation. We are faced once again with essentially the same problem everyone will have seen as a graduate student, with the added complication that one has to keep track of multi-indices up to the second order and, as a result, another constant of integration enters into the solution. In the Cartan formalism (in the author's opinion), the calculation becomes particularly perspicuous, as Thirring demonstrates in vol. II, {\S}4.4, and we shall follow his lead.

\subsubsection{Resolving the Problem with Flatness of Euclidean Space in Radial Coordinates}

Now that we have at hand the elements of an exterior calculus at higher order, we are in a position to follow up on a subtle point left unaddressed in {\S}I.4.6, namely this: we would like to check explicitly that the generalized Riemannian curvature in Euclidean space remains zero under change of coordinate from Cartesian to radial, as it must if the theory is to be consistent. In the conventional differential geometry, this condition boils down to a straightforward calculation, most easily done in Cartan's formalism as we advance it in {\S}I.4.8 above. If, however, one proceeds na{\"i}vely from equations (I.4.61), (I.4.63) or (I.4.72), respectively, a problem will be encountered. For the Euclidean plane in polar coordinates, direct calculation yields for the affine jet connection 1-forms the following:
\begin{align}\label{connection_forms_in_radial_coord_unmodified}
	\omega^r_\theta &= -d^\theta \nonumber \\
	\omega^r_{r\theta} &= -d^{r\theta} \nonumber \\
	\omega^r_{\theta\theta} &= \frac{d^r}{r^2} - \left( 2 r + \frac{1}{r} \right) d^{\theta\theta} \nonumber \\
	\omega^\theta_{r\theta} &= \frac{2 d^r}{r^2} \nonumber \\
	\omega^\theta_{\theta\theta} &= \frac{d^\theta}{r} \nonumber \\
	\omega^{rr}_{\theta\theta} &= - 2 d^{\theta\theta} \nonumber \\
	\omega^{r\theta}_{\theta\theta} &= \frac{d^{r\theta}}{r},
\end{align}
(all others with $\alpha < \beta$ equal to zero). As we see, at second order some of these contain a radial dependence that has little hope of canceling between $\text{\th}\omega^\alpha_\beta$ and $\omega^\alpha_\gamma \wedge \omega^\gamma_\beta$ in order to yield curvature 2-forms $R^\alpha_\beta$ that vanish identically. In fact, they do not cancel (the resulting expressions are too unwieldy to reproduce here). The problem is even more severe when one goes from Cartesian to spherical coordinates in Euclidean 3-space. Must we conclude that equations (I.4.61), (I.4.63) or (I.4.72) are incorrect?

After due deliberation, the solution to this dilemma we should like to propose goes as follows: a satisfactory formalism that retains equations  (I,4,61), (I.4.63) or (I.4.72), in essence, can be found out if we are willing to extend our ring of functions to include infinitesimal elements. We have already contended that in order to arrive at the relevant algebra of observables, one should replace $C^\infty(M)$ with $\mathscr{A}^\infty(M)$---now this replacement, suggested in the first instance on theoretical grounds, begins to make a material difference. Henceforth, we shall regard the space of smooth differential forms $\Omega^*(M)$ as existing over the ground ring $\mathscr{A}^\infty(M)$ by tensoring all modules with $\mathscr{A}^\infty(M)$ over $C^\infty(M)$. To preserve its character as an exterior algebra, we have to define how the exterior derivative $\text{\th}$ is to be understood to act on infinitesimal elements in $\mathscr{A}^\infty(M)$. We require a notation in order to distinguish between a differential when it appears in $\mathscr{A}^\infty(M)$ and when it appears as a basis element of the 1-forms. To this end, we shall write $d^\alpha$ as usual for the former but $\underline{d^\alpha}$ for the latter. Now, every function $f \in \mathscr{A}^\infty(M)$ can be written with respect to a jet basis as a formal power series of the form,
\begin{equation}
	f = \sum_{\alpha \ge 0} f_\alpha d^\alpha,
\end{equation}
where we understand $d^{(0,\ldots,0)} = 1$. By linearity, it suffices to define the exterior derivative on expressions of the form $g d^\alpha$, $g \in C^\infty(M)$, for arbitrary multi-indices $\alpha \in \vvmathbb{N}_0^n$. Thus, let us take
\begin{align}
	\text{\th} ( g 1 ) &= g_{,\beta} \underline{d^\beta} \nonumber \\
	\text{\th} ( g d^\alpha ) &= g_{,\beta} d^\alpha \underline{d^\beta} + g \underline{d^\alpha}, \qquad \alpha > 0.
\end{align}
Compute
\begin{equation}
	\text{\th}^2 (g 1) = \text{\th} \left( g_{,\alpha} \underline{d^\alpha} \right) = g_{,\alpha\beta} \underline{d^\beta} \wedge \underline{d^\alpha} = 0, 
\end{equation}
as usual, but now also (for $\alpha>0$)
\begin{equation}
	\text{\th}^2 (g d^\alpha) = \text{\th} \left( g \underline{d^\alpha} + g_{,\beta} d^\alpha \underline{d^\beta} \right) = g_{,\beta} \underline{d^\beta} \wedge \underline{d^\alpha} + \left( g_{,\beta} \underline{d^\alpha} + g_{,\beta\gamma}d^\alpha \underline{d^\gamma} \right) \wedge \underline{d^\beta} = 0.
\end{equation}
Therefore, $\text{\th}^2=0$ as it should.

If we are willing to admit infinitesimal elements into the ground ring, it gives us a considerable degree of flexibility. In particular, we can eliminate the undesirable terms in equation (\ref{connection_forms_in_radial_coord_unmodified}) by means of a suitable modification. Let us be guided in this process by two goals, first, to employ our freedom in defining the co-frame basis so as to get rid of spatial dependence in the resulting affine connection 1-forms as far as possible, but second, to preserve meanwhile the usual structure in radial coordinates at the 1-jet level, up to an infinitesimal modification (so that in the next subsection, the field equations obtained via Schwarzschild's Ansatz will retain the same form in their leading terms).

The first thing we can do is to make the radial basis element an exact 1-form so that its exterior derivative will vanish automatically. To this end, define
\begin{equation}
F_r = r (1 - d^{\theta\theta}-\sin^2 \theta d^{\phi\phi});
\end{equation}
thence,
\begin{align}
e^r = &\text{\th} F_r \nonumber \\
= &(1 - d^{\theta\theta}-\sin^2 \theta d^{\phi\phi}) \underline{d^r} - 
(r + 2 r \cos^2 \theta d^{\phi\phi} + 2 r \sin^2 \theta d^{\phi\phi}) \underline{d^{\theta\theta}} - r \sin^2 \theta \underline{d^{\phi\phi}} - \nonumber \\
&\sin 2 \theta d^{\phi\phi} ( r \underline{d^\theta} + 2 \underline{d^{r\theta}} ).
\end{align}
Appropriate forms for the other co-frame elements can be determined by experimentation. We merely quote the result here:
\begin{align}\label{radial_coframe_basis_with_inf}
e^\theta = &F_r \underline{d^\theta} + 2 \underline{d^{r\theta}} - \left( (r - d^r) (\sin \theta \cos \theta - \cos 2 \theta d^\theta + 2 \sin 2 \theta d^{\theta\theta}) - \cos 2 \theta d^{r\theta} \right) \underline{d^{\phi\phi}} \nonumber \\
e^\phi = &F_r \left( \sin \theta  - 
\sin \theta \cos^2 \theta  d^{\phi\phi} \right) \underline{d^\phi} + 
2 \left( \sin \theta - \cos \theta d^\theta + 
\sin \theta d^{\theta\theta} \right) \underline{d^{r\phi}} + \nonumber \\
&2 \left( (r - d^r) ( 
\cos \theta + \sin \theta d^\theta + \cos \theta d^{\theta\theta}) +
\sin \theta d^{r\theta} \right) \underline{d^{\theta\phi}} \nonumber \\
e^{rr} = & \underline{d^{rr}} \nonumber \\
e^{r\theta} = &(r - d^r) \underline{d^{r\theta}} \nonumber \\
e^{r\phi} = & \left( (r - d^r) (\sin \theta - \cos \theta d^\theta + 
\sin \theta d^{\theta\theta}) - \cos \theta d^{r\theta} \right) \underline{d^{r\phi}} \nonumber \\
e^{\theta\theta} = & (r^2 - 2 r d^r - 2 d^{rr}) \underline{d^{\theta\theta}} \nonumber \\
e^{\theta\phi} = & \left( (r^2 - 2 r d^r - 2 d^{rr}) ( \sin \theta - 
\cos \theta d^\theta + \sin \theta d^{\theta\theta}) - 
2 r \cos \theta d^{r\theta} \right) \underline{d^{\theta\phi}} \nonumber \\
e^{\phi\phi} = & \left( (r^2 - 2 r d^r - 2 d^{rr}) ( \sin^2 \theta - 
2 \cos \theta \sin \theta d^\theta - 
2 \cos 2 \theta d^{\theta\theta} - 
2 (r - d^r) \sin 2 \theta d^{r\theta} \right) \underline{d^{\phi\phi}}.
\end{align}
As can easily be checked by inspection, equation (\ref{radial_coframe_basis_with_inf}) reduces to equation (I.4.72) upon modding out by the infinitesimal elements. The point of the additional terms is to ensure a simple form for the affine jet connection 1-forms, to be found upon taking the exterior derivative. In fact, we compute in place of equation (\ref{connection_forms_in_radial_coord_unmodified}) the non-zero 1-forms (up to inessential infinitesimal terms) to be just the following:
\begin{align}\label{connection_forms_in_radial_coord_modified}
\omega^r_\theta &= - \underline{d^\theta} \nonumber \\
\omega^r_\phi &= - \sin \theta (1 + d^{\theta\theta}) \underline{d^\phi} \nonumber \\
\omega^\theta_\phi &= - \cos \theta (1 + d^{\theta\theta}) \underline{d^\phi}. \end{align}
These definitions yield the non-zero exterior derivatives and wedge products (up to infinitesimal elements) as follows,
\begin{align}
\text{\th} \omega^r_\phi &= - \cos \theta \underline{d^\theta} \wedge \underline{d^\phi} \nonumber \\
\text{\th} \omega^\theta_\phi &= \sin \theta \underline{d^\theta} \wedge \underline{d^\phi} \nonumber \\
\omega^r_\theta \wedge \omega^\theta_\phi &= \cos \theta \underline{d^\theta} \wedge \underline{d^\phi} \nonumber \\
\omega^\theta_r \wedge \omega^r_\phi &= - \sin \theta \underline{d^\theta} \wedge \underline{d^\phi}.
\end{align}
Thus, the desired cancellation occurs and the curvature 2-forms vanish identically: $R^r_\theta = R^r_\phi = R^\theta_\phi = R^r_{rr} = R^r_{r\theta} = R^r_{r\phi} = R^r_{\theta\theta} = R^r_{\theta\phi} = R^r_{\phi\phi} = R^\theta_{rr} = R^\theta_{r\theta} = R^\theta_{r\phi} = R^\theta_{\theta\theta} = R^\theta_{\theta\phi} = R^\theta_{\phi\phi} = R^\phi_{rr} = R^\phi_{r\theta} = R^\phi_{r\phi} = R^\phi_{\theta\theta} = R^\phi_{\theta\phi} = R^\phi_{\phi\phi} = R^{rr}_{r\theta} = R^{rr}_{r\phi} = R^{rr}_{\theta\theta} = R^{rr}_{\theta\phi} = R^{rr}_{\phi\phi} = R^{r\theta}_{r\phi} = R^{r\theta}_{\theta\theta} = R^{r\theta}_{\theta\phi} = R^{r\theta}_{\phi\phi} = R^{r\phi}_{\theta\theta} = R^{r\phi}_{\theta\phi} = R^{r\phi}_{\phi\phi} = R^{\theta\theta}_{\theta\phi} = R^{\theta\theta}_{\phi\phi} = R^{\theta\phi}_{\phi\phi} = 0$ (cf. Thirring, \cite{thirring_vol_2}, {\S}4.4.32).

At this point, a subtlety that does not exactly leap out of the formulae bears mentioning. In order for the coefficients to work out nicely, we must adopt a convention as to the meaning of the sum in equation (I.4.87). Here, we regard the sum as being taken over all combinations of $r$-tuples in the $n$ variables at each successive order. If, however, the sum is to be regarded instead as being taken over unique multi-indices, a corresponding combinatorial factor has to be included. That is, one may sum over unique multi-indices if at the same time the partial derivative operator is defined to be,
\begin{equation}
\partial_\alpha := \frac{|\alpha|!}{\alpha!} \frac{\partial^{|\alpha|}}{\partial x_1^{\alpha_1} \cdots \partial x_n^{\alpha_n}}.
\end{equation}
One checks that either way $\text{\th}^2 = 0$.

\subsubsection{Schwarzschild Solution to Second Order in the Jets}

Following Thirring's method, we are prepared to derive the Schwarzschild solution to second order in post-Einsteinian general theory of relativity. We may presume that the spherical symmetry of the problem should lead to a static solution having spatial dependence only on the radial coordinate. Moreover, physical intuition suggests that there should be a resemblance to what happens in the 1-jet case, namely, that the curvature should arise only in the temporal and radial directions while the transverse directions (depending on latitude and longitude on the two-dimensional spheres at a given radius and time) should be unaffected. In view of equation (\ref{radial_coframe_basis_with_inf}) therefore, posit an Ansatz of the form (here, for the sake of simplicity we write out only the non-infinitesimal terms)
\begin{align}
	e^t &= e^a \underline{d^t} \nonumber \\
	e^r &= e^b \left( \underline{d^r} - \frac{1}{2} r \underline{d^{\theta\theta}} - \frac{1}{2} r \sin^2 \theta \underline{d^{\varphi\varphi}} \right) \nonumber \\
	e^\theta &= r \underline{d^\theta} + \underline{d^{r\theta}} - \frac{1}{2} r \sin \theta \cos \theta \underline{d^{\varphi\varphi}} \nonumber \\
	e^\varphi &= r \sin \theta \underline{d^\varphi} +  \sin \theta \underline{d^{r\varphi}} +  r \cos \theta \underline{d^{\theta\varphi}} \nonumber \\
	e^{tt} &= e^{2a_1} \underline{d^{tt}} \nonumber \\
	e^{tr} &= e^{a_1+b_1} \underline{d^{tr}} \nonumber \\
	e^{t\theta} &= r \underline{d^{t\theta}} \nonumber \\
	e^{t\varphi} &= r \sin \theta \underline{d^{t\varphi}} \nonumber \\
	e^{rr} &= e^{2b_1} \underline{d^{rr}} \nonumber \\
	e^{r\theta} &= r \underline{d^{r\theta}} \nonumber \\
	e^{r\varphi} &= r \sin \theta \underline{d^{r\varphi}} \nonumber \\
	e^{\theta\theta} &= r^2 \underline{d^{\theta\theta}} \nonumber \\
	e^{\theta\varphi} &= r^2 \sin \theta \underline{d^{\theta\varphi}} \nonumber \\
	e^{\varphi\varphi} &= r^2 \sin^2 \theta \underline{d^{\varphi\varphi}},
\end{align}
where $a$, $b$, $a_1$ and $b_1$ depend on radius only. One calculates from this Ansatz the following non-zero affine jet connection 1-forms:
\begin{align}
\omega^t_r &= e^{a-b} a^\prime \underline{d^t} \nonumber \\
\omega^r_\theta &= - e^{-b} \underline{d^\theta \nonumber} \\
\omega^r_\varphi &= - e^{-b} \sin \theta (1 + d^{\theta\theta}) \underline{d^\varphi} \nonumber \\
\omega^\theta_\varphi &= - \cos \theta (1 + d^{\theta\theta}) \underline{d^\varphi} \nonumber \\
\omega^r_{tt} &= -2 e^{2 a_1 - b} a_1^\prime \underline{d^{tt}} \nonumber \\
\omega^{tt}_{rr} &= 2 e^{2 a_1 - 2 b_1} (a_1^{\prime\prime} + 2 a_1^{\prime 2} ) \underline{d^{tt}} \nonumber \\
\omega^{tt}_{\theta\theta} &= \omega^{tt}_{\varphi\varphi} = \frac{2 e^{2 a_1} a_1^\prime}{r} \underline{d^{tt}} \nonumber \\
\omega^r_{tr} &= - e^{a_1 + b_1 - b} (a_1^\prime + b_1^\prime) \underline{d^{tr}} \nonumber \\
\omega^r_{rr} &= - 2 e^{2 b_1 - b} b_1^\prime \underline{d^{rr}} \nonumber \\
\omega^{tr}_{rr} &= e^{a_1 - b_1} \left( a_1^{\prime\prime} + b_1^{\prime\prime} + (a_1^\prime + b_1^\prime)^2 \right) \underline{d^{tr}} \nonumber \\
\omega^{tr}_{\theta\theta} &= \omega^{tr}_{\varphi\varphi} = \frac{e^{a_1 + b_1} (a_1^\prime + b_1^\prime)}{r} \underline{d^{tr}} \nonumber \\
\omega^{rr}_{\theta\theta} &= \omega^{rr}_{\varphi\varphi} = \frac{2 e^{2 b_1}b_1^\prime}{r} \underline{d^{rr}}.
\end{align}

If the curvature 2-forms be computed according to the above formula, in general, additional terms involving higher derivatives of $a$ and $b$ will supervene (compared to the formulae quoted by Thirring). For instance,
\begin{align}\label{schwarzschild_01}
	R^t_r = - &e^{a-b} (a^{\prime 2} -  
	a^\prime b^\prime + a^{\prime\prime}) \underline{d^t} \wedge \underline{d^r} - \nonumber \\ &e^{a-b} ( 
	a^{\prime 3} - 2 a^{\prime 2}b^\prime - 2 a^{\prime\prime}b^\prime + a^\prime b^{\prime 2} + 3 a^\prime a^{\prime\prime} - a^\prime b^{\prime\prime} + a^{\prime\prime\prime} ) \underline{d^t} \wedge \underline{d^{rr}}. 
\end{align}
The complete expressions for the curvature 2-forms in the 2-jet case are too unwieldly to write out in full, but one sees that as expected one gets a perturbation from the 1-jet result, if we substitute
\begin{equation}\label{schwarzschild_a_b}
	a = \frac{1}{2} \log \left( 1 - \frac{A_0}{r} \right) \qquad
	b = -\frac{1}{2} \log \left( 1 - \frac{B_0}{r} \right).
\end{equation}
Then equation (\ref{schwarzschild_01}) becomes,
\begin{equation}
	R^t_r = \frac{A_0}{r^3} \underline{d^t} \wedge \underline{d^r} + \frac{3(A_0-B_0)}{4 r^4} \underline{d^t} \wedge \underline{d^r} - \frac{3A_0}{r^4} \underline{d^t} \wedge \underline{d^{rr}}.
\end{equation}
Let us explain now how the field equations appear at this level in the jets. If one bear in mind the block-diagonal structure of the Hodge-*operator, it is apparent that the field equations for $|\alpha|=1$ will be unaffected, as far as their dependence on $K^\alpha_\beta$ goes but, of course, these latter quantities will have the second-order corrections built in; namely:
\begin{equation}\label{stress_energy_curvature}
	8 \pi \varkappa t_\alpha = - \sum_{\beta<\gamma, \beta \ne \alpha, \gamma \ne \alpha} K_{\beta\gamma} e^\alpha.
\end{equation}
Here, we write the diagonal part of the curvature as $K_{\beta\gamma}$, where $R_{\beta\gamma} = K_{\beta\gamma} e^\beta \wedge e^\gamma$ (no sum) plus any off-diagonal part, which we presume to be negligible as it first enters at subleading order (i.e., proportional to $1/r^4$). Consider now a source at rest and assume stress-energy forms going like
\begin{equation}
	t^\alpha = (\varrho e^t,0,0,0,\sigma e^{tt},0,0,0,- \sigma e^{rr},0,0,0,0,0).
\end{equation}
Equation (\ref{schwarzschild_a_b}) with $r_0 = 2GM$ (where $M=4\pi\int \varrho r^2 d^r$) solves the first four field equations in the 1-jets, as usual. Now, $\sigma$ yields another constant of integration that may appear once 2-jets are included in the picture. Let $r_1=2GN$ with $N=4\pi\int \sigma r^2 d^r$. If we adopt an Ansatz of the following kind for the functions $a_1$ and $b_1$:
\begin{equation}\label{schwarzschild_a1_b1}
	a_1 = \frac{1}{2} \log \left( 1 - \frac{A_1}{r} \right) \qquad
	b_1 = -\frac{1}{2} \log \left( 1 - \frac{B_1}{r} \right),
\end{equation}
then a suitable combination of the coefficients should balance on the left hand side the contribution to the 2-jet stress-energy forms on the right hand side arising from non-zero $N$, just as the usual Schwarzschildian functions $a$ and $b$ in equation (\ref{schwarzschild_a_b}) balance the contribution to the 1-jet stress-energy arising from non-zero $M$. In fact, the condition for the leading $1/r^3$ dependence on the left hand side of the field equations to cancel is just $A_0-3A_1-B_0+3B_1=0$. Thus, we could take $A_0=r_0$, $B_0=r_0$, $A_1=r_0+r_1$ and $B_1=r_0+r_1$ in order to recover the 1-jet result along with a perturbation corresponding to non-zero $r_1$. This finding suggests that a solution to the full problem exists and could be found out iteratively through imposition of cancellation term by term in a further Ansatz of the form,
\begin{align}\label{schwarzschild_a1_b1_a2_b2}
	a &= \frac{1}{2} \log \left( 1 - \frac{A_{01}}{r} - \frac{A_{02}}{r^2} - \cdots \right) \qquad
	b = -\frac{1}{2} \log \left( 1 - \frac{B_{01}}{r} - \frac{B_{02}}{r^2} - \cdots \right) \nonumber \\	
	a_1 &= \frac{1}{2} \log \left( 1 - \frac{A_{11}}{r} - \frac{A_{12}}{r^2} - \cdots \right) \qquad
	b_1 = -\frac{1}{2} \log \left( 1 - \frac{B_{11}}{r} - \frac{B_{12}}{r^2} - \cdots \right).
\end{align}
To do better than such an Ansatz, one would have to know the full exact solution for the curvature 2-forms. Evidently, an improved understanding of the Schwarzschild solution at second order would pose an excellent test of the consistency of the whole formalism of the present work.

\subsection{Cosmological Solutions}
\label{cosmological_model}

In keeping with the philosophy behind the present work, we wish merely to adumbrate a toy model as a simple illustration of the theory and to indicate one direction in which post-Einsteinian effects may supervene. Therefore, let us posit uniformity, homogeneity and isotropy and seek a solution of Friedmann-Robertson-Walker form (see \cite{friedmann}, \cite{robertson}, \cite{walker}) to $r=2$; cf. Thirring, vol. 2, {\S}4.4.

The Ansatz for the orthonormal frame to use in this case is as follows:
\begin{align}
	e^t &= \underline{d^t} \nonumber \\
	e^\mathbf{x} &= \frac{R}{1+Kr^2/4} \underline{d^\mathbf{x}} \nonumber \\
	e^{tt} &= \underline{d^{tt}} \nonumber \\
	e^{t\mathbf{x}} &=  \frac{R}{1+Kr^2/4} \underline{d^{t\mathbf{x}}} \nonumber \\
	e^{\mathbf{x}\mathbf{y}} &=  \frac{R^2}{(1+Kr^2/4)^2} \underline{d^\mathbf{xy}},
\end{align}
where $R$ is a function of $t$ only, $r=\sqrt{x^2+y^2+z^2}$ and $K$ is a constant parametrizing the spatial curvature. The point is that we may evaluate the curvature 2-forms in terms of $R$ and $K$, then substitute into the field equations with a suitable phenomenological Ansatz for the stress-energy forms. In view of the factorization property for the 1-jet solution (cf. Thirring, \cite{thirring_vol_2}, {\S}4.4, in particular, Equation 4.4.6), we must have a leading dependence of the $t_\alpha$ proportional to some function of time multiplied with $e^\alpha$ and any 2-jet corrections will enter at subleading order.

We shall not follow the problem through all the way to a solution of the cosmological model, but merely indicate the nature of how it changes due to the incorporation of higher order jets. The temporal-spatial curvature 2-forms turn out to be no different from the 1-jet case if we ignore the off-diagonal terms:
\begin{equation}
	R^t_\mathbf{x} = \frac{\ddot{R}}{R} e^t \wedge e_\mathbf{x}
\end{equation}
but besides $R^t_{tt} = 0$ we have additional ones of the form,
\begin{align} 
	R^t_{t\mathbf{x}} &= \frac{\ddot{R}}{R} e^t \wedge e_{t\mathbf{x}} \nonumber \\
	R^t_\mathbf{xy} &= 2 \frac{\dot{R}^2 + R \ddot{R}}{R^2} e^t \wedge e_\mathbf{xy}.
\end{align}
A novel complication enters, however, with the spatial-spatial curvature 2-forms. Due to the possibility of taking higher derivatives with respect to spatial coordinates, the resulting expressions no longer reduce to simple functions of $r$. The full outcome is too long to quote, but we show what one gets as an expansion in $K$ up to second order:
\begin{align}
	R^\mathbf{x}_\mathbf{y} &= \left( \frac{K+\dot{R}^2-\ddot{R}^2}{R^2} + K^2 \frac{r^2 \dot{R}^2-3}{4 R^4} \right) e^\mathbf{x} \wedge e_\mathbf{y} \nonumber \\
	R^\mathbf{x}_{tt} &= - \frac{R^{(\mathrm{iv})}}{R} e^\mathbf{x} \wedge e_{tt} \nonumber \\
	R^\mathbf{x}_{t\mathbf{y}} &= \left( \frac{\dot{R}^2-\ddot{R}^2}{R^2} + K \frac{R^2-4\dot{R}^2+4R\ddot{R}}{2R^4} + K^2 \frac{\frac{1}{2}r^2 \dot{R}^2+\frac{1}{2}(r^2-2x_i^2) R \ddot{R}-\frac{3}{4}-\frac{1}{8}r^2 R^2}{R^4} \right) e^\mathbf{x} \wedge e_{t\mathbf{y}} \nonumber \\
	R^\mathbf{x}_{\mathbf{y}\mathbf{z}} &= \left( 2 \frac{R\dot{R}^2-\dot{R}^2\ddot{R}-R\ddot{R}^2}{R^3} + \frac{K}{R^2} + K^2 \frac{4r^2 \dot{R}^2-\frac{1}{4}(r^2+2x_i^2)R^2-\frac{7}{2}}{R^4} \right) e^\mathbf{x} \wedge e_\mathbf{yz}
\end{align}
At this juncture, it would become too laborious to quote formulae for the rest of the curvature 2-forms since they are similar in nature to what has already been presented.

It is worth indicating what becomes of the field equations themselves in the Friedmann-Robertson-Walker model to second order (where we ignore as negligible any direct contributions from the 2-jet terms; i.e., we consider only their indirect effect through the perturbation they induce in the 1-jet terms\footnote{This manner of proceeding can be justified by reasoning as follows: in order to specify the model fully, one has to introduce a characteristic scale $\ell$ to govern the relative sizes of the infinitesimals at higher order. Then the $K_{\beta\gamma}$ in equation (\ref{stress_energy_curvature}) should be multiplied by a factor of $\left( \dfrac{\ell}{b} \right)^{|\beta|+|\gamma|-2}$, where $b$ is the fundamental unit of length. As long as $\ell \ll b$ the 2-jet terms will be suppressed by a small coefficient. A reasonable suggestion would be to take for $\ell$ the Planck length, $\ell_P$.}):
\begin{align}
	3 \frac{\dot{R}^2 + K - \ddot{R}^2}{R^2} - K^2 \frac{9-12 r^2 \dot{R}^2}{4R^4} &= 8 \pi \varkappa \varrho \nonumber \\
	- \frac{\dot{R}^2+K+2R\ddot{R}-2\ddot{R}^2}{R^2} + K^2 \frac{3-4 r^2 \dot{R}^2}{4R^4} &= 8 \pi \varkappa p.
\end{align}
Needless to say, dropping terms in $\ddot{R}^2$ and $K^2$ these formulae return to the usual 1-jet Friedmann-Robertson-Walker model (cf. Thirring \cite{thirring_vol_2}, {\S}4.4). 

We leave an investigation of the astrophysical significance of our toy model, as well as of allied questions such as the existence of Killing higher tangent vector fields, to the experts. The theory developed here naturally suggests a research program for investigating the acceleration of the rate of expansion of the universe, first observed in 1998 by the High-z Supernova Search Team and the Supernova Cosmology Project, \cite{reiss_et_al} resp. \cite{perlmutter_et_al}. Can the tension between measurements of the Hubble constant based on the flow of nearby galaxies in the local universe versus those based on the characteristics of the cosmic background radiation during the era of decoupling (at cosmological distances) thereby be resolved? For a recent review of the astronomical issues involved, consult Shah et al., \cite{shah_lemos_lahav}.

	\section{Novel Effects in Orbital Mechanics}\label{chapter_7}

The formalism of the general theory of relativity in the presence of higher infinitesimals set forth in {\S\ref{chapter_6} may be brought to bear on two astronomical contexts which offer us a chance to confront the present theory with observation: the Pioneer and flyby anomalies. Supposing them not to be spurious artifacts produced by systematic errors in the data reduction, they offer a clue into the nature of post-Einsteinian physics. In our interpretation, the reason behind the anomalous effect, in both cases, may be attributed to subtleties involved in passing from one coordinate system to another, which do not show up at linear order in the infinitesimals. Therefore, they are inexplicable on the basis of the general theory of relativity as currently construed, in the form in which Einstein left it. For spacecraft on an escape trajectory from the solar system, we find a small centripetal acceleration, reflecting the modified inertia and depending on one free cosmological parameter. For spacecraft on a hyperbolic orbit passing near a planet, an inertial frame dragging linear in velocity gives rise to a Coriolis force at second order, leading to the flyby anomaly. The semi-empirical prediction formula due to Anderson and coworkers is derived.

\subsection{Pioneer Anomaly}

Anderson et al. \cite{anderson_Pioneer} report a residual centripetal acceleration being experienced by the Pioneer spacecraft, currently on an escape trajectory from the solar system since its encounter with Jupiter in 1973, on the order of $8.5 \times 10^{-8} ~\mathrm{cm} ~\mathrm{s}^{-2}$. The effect is close to being constant in magnitude over distances from the sun ranging from 20 AU to 95 AU. The Pioneer anomaly has eluded explanation in terms of known physics.\footnote{Thus, one finds questionable two aspects of the analysis by Turyshev et al. \cite{turyshev}, who dismiss the anomaly as being due to radiation pressure from anisotropic thermal emission from the spacecraft's radioisotopic generator and heated electrical components. These authors account for only 80 percent of the reported anomalous acceleration in the first place and second, their model requires the effect to decline exponentially in time with a half-life of forty years, in contradiction with Anderson et al.'s finding that it is nearly constant. In any case, the result of this section showing that an anomalous centripetal acceleration ought to arise very naturally in post-Einsteinian theory demands a proper reanalysis to disentangle the contribution of new physics from that of anisotropic thermal radiation.} Hence, we wish to investigate what, if anything, the post-Einsteinian theory has to say about it. Our circumscribed purpose in this section is to apply the Schwarzschild solution obtained in {\S}\ref{schwarzschild_solution} to the problem of inertial motion in almost flat space, in order to check whether any corrections of possible relevance to the Pioneer anomaly might enter. As we shall see, it can be viewed as an implication of modified inertia in post-Einsteinian theory.

For the arclength of a world-line, or what amounts to the same thing, the proper time experienced by a spacecraft is to be computed by integration of the induced Riemannian volume form which now contains terms of higher order. As discussed above in {\S}\ref{lorentz_to_higher_order}, a different measurement procedure applies to bodies in unrestrained inertial motion than does to the rigid rods and stationary clocks by which we fix our reference frame. The jet geodesic equation can be expressed in terms of any coordinates we please. It is convenient to start out with the inertial frame with respect to which the center of mass of the solar system is at rest. The time coordinate $t$ measures time elapsed on stationary clocks located on earth (or, in principle, anywhere in the solar system). After setting up the equation of motion, it makes sense to go over to the spacecraft's proper time. Why? We do not have an observer stationed on the spacecraft who can note the position and time indicated on the stationary clocks, as he passes by them. Rather, what one measures is the Doppler shift of the radio signal from the spacecraft's antenna. Now, the increment in proper time between successive signals emitted from the spacecraft will be the same as what we record when they arrive at earth. Since we do have stationary clocks on earth to refer to and we know our position in the solar system in terms of the earth's orbital elements, we can work out what the relative velocity between us and the spacecraft is, and thence infer its component of acceleration along the line of sight as a function of time.

The upshot of our discussion is that we have first to formulate the jet geodesic equation with respect to ephemeris time $t$, then go over to the spacecraft's proper time $s$. The procedure is non-trivial due to the corrections at second order in the jets that enter now in post-Einsteinian general theory of relativity.

\subsubsection{Anomalous Centripetal Acceleration of Spacecraft on an Escape Trajectory}\label{anomalous_centripetal_derivation}

As can be seen from Figure 3 in \cite{anderson_Pioneer}, since its encounter with Jupiter on December 4, 1973, Pioneer 10 has been traveling on an escape trajectory from the solar system directed approximately radially outwards from the sun. To get a feeling for what happens, we reduce the problem to a single dimension and suppose that its core physics can be captured by treating the Pioneer spacecraft as if it were on a strictly radial outward trajectory, starting from its passage past Jupiter at 5.2 astronomical units from the sun. The relation between its initial outward velocity $\varv$ at radius $r$ and its asymptotic velocity $\varv_\infty$ may be found from the condition of conservation of energy:
\begin{equation}
	\frac{1}{2} m_0 \varv^2 - \frac{G M m_0}{r} = \frac{1}{2} m_0 \varv_\infty^2;
\end{equation}
whence,
\begin{equation}\label{velocity_versus_radius}
	\varv = \sqrt{\varv_\infty^2+\frac{2GM}{r}}.
\end{equation}
The time $t$ to reach radius $r$ may then be determined by integration:
\begin{equation}\label{time_to_radius}
	t - t_0 = \int_{r_0}^r \frac{dr}{ \sqrt{\varv_\infty^2+{2GM}/{r}}}.
\end{equation}
In principle, this relation may be inverted to obtain $r$ as a function of time (although the integral does not exist in closed analytic form). The proper time experienced by the spacecraft, however, is not $t$ (which we measure using clocks close to being at rest), but in the non-relativistic limit neglecting Lorentz factors
\begin{equation}
	s = \int \theta^t = \int (d^t + d^{tt} + \cdots) = e^t - 1 = t + \frac{1}{2}t^2 + \cdots
\end{equation}
whence (inserting the dimensional factors, where for convenience we denote the cosmological scale factor by $b$),
\begin{equation}
	\dot{s} = 1 + \frac{ct}{b} + \cdots; \qquad \frac{1}{\dot{s}} = 1 - \frac{ct}{b} + \cdots
\end{equation}
For the remainder of the calculation, we presume the time scale involved to be short compared to $b/c$ so that we may retain only the linear term in the exponential decay of $1/\dot{s}$. From equation (I.2.8) (if it is not obvious), we may write
\begin{equation}
	\partial_t = \dot{s} \partial_s = (1 + t) \partial_s.
\end{equation}
As shown above in {\S}I.3, in the non-relativistic regime in the weak-field limit the jet geodesic equation reduces to Newton's law, which we may state as follows:
\begin{equation}
	\partial_t \partial_t x^\mu = - \Gamma^\mu_{00}.
\end{equation}
Therefore, upon change of time coordinate the left hand side reads,
\begin{equation}
	\partial_t (1+t) \partial_s x^\mu = (1+t) \partial^2_s x^\mu + \partial_s x^\mu. 
\end{equation}
For the sake of notational simplicity, write the spacecraft's velocity as a function of radial distance from the sun as $\varv$ and its acceleration as $a$. To leading order, we obtain
\begin{equation}
	a = \partial_s^2 x^\mu = - \frac{c\varv}{b} - \frac{GM_\odot}{r^2(1+ct/b)}.
\end{equation}
Therefore, we have for the anomalous acceleration of the spacecraft coming from the the post-Einsteinian modification of inertia the leading dependence,
\begin{equation}\label{anomalous_accel_full}
	a_1 = - \frac{c\varv}{b} + \frac{G M_\odot ct}{b r^2}.
\end{equation}
The anomalous acceleration of equation (\ref{anomalous_accel_full}) will always be centripetal (since for a spacecraft on an escape trajectory its total speed will be monotonically decreasing and therefore will satisfy 
\begin{equation}
	\frac{t}{r} = \frac{t}{r_0 + \int_0^t \varv dt} < \frac{t}{r_0+\varv t} < \frac{1}{\varv};
\end{equation}
moreover it must be greater than the velocity corresponding to a circular orbit at the given radius). At long enough times, however, by which epoch the spacecraft will nearly have reached its asymptotic velocity $\varv_\infty = 12.2 ~\mathrm{km/s}$, we can substitute the approximation $t=r/\varv_\infty$ and the anomalous acceleration it experiences goes over into the limiting form
\begin{equation}\label{anomalous_accel}
	a_1 = - \frac{c\varv_\infty}{b} + \frac{G M_\odot c}{b r \varv_\infty},
\end{equation}
displaying a dependence on radius as the inverse first power. Let us enter now into a quantitative discussion and comparison with observations. 

The agreement thus obtained with the Caltech group's findings is remarkable, if imperfect. The derivation of equation (\ref{anomalous_accel}) involves a somewhat drastic simplification of the actual situation. A proper analysis would have to solve the relevant jet geodesic equation, which we are not close to be being able to do at the moment. Since the time to reach radius $r$ from equation (\ref{time_to_radius}) will be materially less than the approximation employed, $r/\varv_\infty$ that is, the initial decline of the anomalous acceleration will be slower than formula (\ref{anomalous_accel}) suggests and accordingly it may be expected that a more accurate solution will fit more closely with observations. 

A somewhat more realistic model would be to solve for $\varv$ and $t$ as functions of $r$ from equations (\ref{velocity_versus_radius}) and (\ref{time_to_radius}). A slight subtlety crops up here so let us indicate how we intend to resolve it. Equation (\ref{anomalous_accel_full}), as it stands, involves an inconsistency. From a physical point of view, it should not matter when one chooses to be the origin of time. The underlying problem here is that $d^s$ does not have compact support (cf. remark I.5.2). A proper treatment of the initial conditions in the jet geodesic equation should lead to a smoothing out of the dip to zero and rapid rise of the second term in the anomalous acceleration at the starting point of the spacecraft's escape trajectory. Since we have not as yet developed sufficiently proper technique for solution of the jet geodesic equation from arbitrary initial conditions, let us resort to a simple stratagem: treat the origin of time $t_0$ as a free parameter and replace $t$ by $t-t_0$ in all formulae. To justify the procedure, consider that after its encounter with Jupiter the spacecraft is nearly in free motion, for which proposition I.3.4 characterizes the solution. In analogy with what is done when defining the M{\o}ller transformations in classical scattering theory (cf. Thirring, \cite{thirring_vol_1}, {\S}3.4), the trajectory of the spacecraft at late times behaves \textit{as if} it started on its inertial motion in free space at time $t_0$, even though this may be supposed to have taken place before its passing Jupiter. 

In figure \ref{pioneer_anomaly_figure}, which is to be compared with Figure 7 in \cite{anderson_Pioneer}, we plot the result of numerical integration for the time to radius $r$ starting from Jupiter at 5.2 astronomical units (assuming $t_0$ to occur half a year before), showing near quantitative agreement and as expected a slower than $1/r$ fall off with radius. Here, to render the result quantitative, the fundamental unit of length $b$ is taken to be as determined in {\S}\ref{determination_of_fundamental_length} below. It is evident that the numerical model of the anomalous acceleration effect is still rather too simplified. In particular, equation (\ref{anomalous_accel_full}) suggests that the dominant component of the anomalous acceleration at late times (though still small compared to $b/c = 1.47 \times 10^{13} ~\mathrm{s}~ = 464,000 ~\mathrm{year}$) must be directed opposite to its forward velocity and not necessarily towards the sun. If one examines Figure 3 in \cite{anderson_Pioneer}, one readily checks that the base of the triangle extends outwards to approximately $7.5 ~\mathrm{AU}$ for Pioneer 10 when it is at distances greater than about $20 ~\mathrm{AU}$ and to approximately $15 ~\mathrm{AU}$ for Pioneer 11 when it is at distances greater than about $30 ~\mathrm{AU}$. Thus, despite everything our simplified model constitutes a reasonable approximation, in that the cosine of the angle between minus its velocity and the radius vector pointing to the sun will range from 0.9 to 0.95, so that any difference will be well within experimental error.

\begin{figure}\label{plot_of_anomalous_accel}
	\includegraphics[width=10cm]{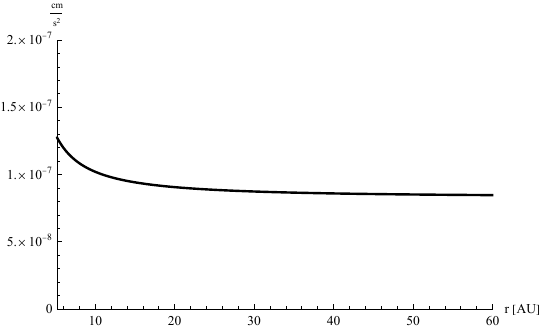}
	\caption{Anomalous acceleration on Pioneer 10 according to equation (\ref{anomalous_accel_full}) (based on the approximate value of the fundamental unit of length to be derived in {\S}\ref{determination_of_fundamental_length}).}
	\label{pioneer_anomaly_figure}
\end{figure}

\begin{remark}
	To the author's knowledge, the empirical grounding of the Pioneer anomaly has never been adequately clarified, due to confirmation bias among those who are convinced that there could be no component of physics beyond the standard model involved and that therefore ascribe all anomalous acceleration to pressure of thermal radiation emitted by the spacecraft's radioisotopic generator. Early work of the Caltech group in which the potential anomaly is identified has not been satisfactorily refuted. Now that a specific analytical form of the former, coming from the jet geodesic equation in the Schwarzschild solution, is available, the possibility of a direct comparative study between the contributions due to novel inertial effects and to radiation pressure exists. 
\end{remark}

\subsubsection{Determination of the Fundamental Unit of Length}\label{determination_of_fundamental_length}

From dimensional considerations alone, one is free to complement the relativistic system of units (in which $c=G=1$) with a unit of length rather than of action, as is usually done. Setting the Caltech group's latest estimate of $a_P = (8.74 \pm 1.33) \times 10^{-8} ~\mathrm{cm} ~\mathrm{s}^{-2}$ (see \cite{anderson_Pioneer}) to $a_1$ in equation (\ref{anomalous_accel}) and remembering that one wants to use the mass $M_\odot$ for the sun, we obtain the following value for the fundamental unit of length at a nominal radius of $30$ astronomical units (where the data are best):
\begin{equation}
	b = 4.40 \times 10^{23} ~\mathrm{cm} = 142 ~\mathrm{kiloparsec}
\end{equation}
(where $G = 6.67408 \times 10^{-8} ~\mathrm{cm}^3 ~\mathrm{g}^{-1} ~\mathrm{s}^{-2}$, $M_\odot = 1.9891 \times 10^{33} ~\mathrm{g}$, $c = 2.99792458 \times 10^{10} ~\mathrm{cm} ~\mathrm{s}^{-1}$, $1 ~\mathrm{AU} = 1.49597892 \times 10^{13} ~\mathrm{cm}$ and $1 ~\mathrm{parsec} = 3.08567802 \times 10^{18} ~\mathrm{cm}$). Needless to say, the estimate is very approximate only but in principle one could derive the exact solution to the jet geodesic equation and fit it to the data to arrive at a better value for $b$.

\subsubsection{Modified Newtonian Phenomenology}\label{MOND}

As we saw in the previous subsection, the Pioneer anomaly appears to involve a post-Einsteinian transition from the Newtonian law of universal gravitation in the environs of the solar system, with its familiar $1/r^2$ dependence on radius, to a novel force law that holds far away and displays a $1/r$ dependence on radius (but which decays to zero at very long times). Something like this behavior is a feature of Milgrom's modified Newtonian dynamics, proposed to explain the problem of apparent missing mass in galaxies and galaxy clusters \cite{milgrom}. Milgrom's idea is that massive bodies moving under the influence of gravitational attraction from distant masses in a regime of very low accelerations bear a modified inertial response to the impressed force, such that a greater acceleration results than would be the case with the ordinary Newtonian second law (in the non-relativistic limit which is applicable to all the astrophysical contexts studied so far). Milgrom's phenomenological Ansatz is by design quite successful in reproducing the flat rotation curves observed in spiral galaxies and is supported by heuristic semi-empirical reasoning, though not by any underlying theory. For a review of the extensive work that has been done on modified Newtonian dynamics since its inception forty years ago, see \cite{mond_review}.

Now, we have just found what appears to be a similar modification to inertia in the very low-acceleration regime. Indeed, the near coincidence between the anomalous acceleration of the Pioneer spacecraft and Milgrom's surmise for $a_0$ is suggestive. Moreover, our result flows from a proposed general theoretical framework of great naturality and is fully covariant, in principle. Nevertheless, there are two salient differences from Milgrom's modified Newtonian mechanics: first, our force law in the low-acceleration regime of interstellar space has a coefficient linear in $GM$ while Milgrom's goes as $\sqrt{GM}$; and second, we have derived our result only in the case of a nearly radial hyperbolic orbit away from a central point mass. In the scenario of modified Newtonian dynamics, one ordinarily presumes it to apply to all orbital motion in the non-relativistic regime, whether bounded or unbounded. In the astrophysical setting of galaxies and clusters of galaxies, observational data are available only for stars resp. galaxies presumably on bound orbits, of course.

Therefore, considerable hard work remains to investigate the possible connection between our post-Einsteinian theory and the modified Newtonian framework. Fortunately, there is for us an obvious starting point when treating bound orbits, namely, the jet geodesic equation once again. Remember, as well, that what has to be recovered are the astronomical data for rotation curves of spiral galaxies and velocity dispersion in galaxy clusters, not necessarily Milgrom's specific proposal for the form the modification of inertia takes, per se.

\subsection{Flyby Anomaly}\label{flyby}

Another deviation from Newtonian dynamics is more decisive. The flyby anomaly consists in a small differential between ingoing and outgoing asymptotic velocities of a spacecraft transiting near to the earth, apparently due to the earth’s rotation \cite{anderson_flyby}. There have been no plausible proposals for how to explain it until now. Nonetheless, J.D. Anderson and his Caltech group \cite{anderson_flyby} have studied its systematics and arrived at the following semi-empirical prediction formula:
\begin{equation}\label{flyby_prediction_formula}
	\frac{\Delta\varv_\infty}{\varv_\infty} = \frac{2 \omega_\oplus R_\oplus}{c} \left( \cos \delta_0 - \cos \delta_1 \right),
\end{equation}
where $\omega_\oplus$ is the rotational velocity of the earth and $R_\oplus$ its radius. Here, $\delta_0$ designates the incoming declination and $\delta_1$ the outgoing declination. 

The flyby anomaly arises as a natural consequence of covariance as applied to the generalized metric tensor in a rotating frame---as such, no contrived hypotheses are required to explain it. For, as Newton stresses in the \textit{Principia}, there are strong reasons to suppose the rotation of a material body relative to an external inertial frame to bear with it observable physical consequences (the spinning bucket and coupled globes thought-experiments). The rotation can be canceled out, of course, by going to an appropriate non-inertial co-rotating frame and the problem of nearby orbital motion solved as if the larger body were stationary. If one then transfer back to the original inertial frame, one expects a departure from non-rotating behavior there which may in general be attributed to so-called Coriolis forces. In the case of the known general theory of relativity, precisely this procedure gives rise to the Lense-Thirring effect, which one is accustomed to describe as being due to inertial frame dragging \cite{einstein_formale_grundlage}, \cite{lense_thirring}.

Now, the transformation to the rotating frame is non-linear and, in particular, involves non-vanishing second derivatives. In view of the more complicated transformation law that applies to generalized tensors, it is natural then to wonder whether the additional terms appearing in the post-Einsteinian theory can manifest themselves as a physically observable effect. Indeed, we shall see in a moment that this eventuality proves to be the case. What is of greatest moment, moreover, is that, contrary to what happens in the Einsteinian case, where terms linear in the rotational velocity appearing as components in the metric tensor enter into observable phenomena only after combining together (causing the effect to be quadratic in the rotational velocity and thus so small as almost to escape detection), in the post-Einsteinian case it is possible for them to enter linearly into the jet geodesic equation. Therefore, we may speak of a hitherto-unknown inertial frame dragging linear in velocity. If the effect were to be only quadratic in velocity, it would remain much too tiny to measure.

\subsubsection{Derivation of the Anomalous Velocity Differential}

There are two issues that demand attention: first, how to set up the coordinate transformation to the co-rotating frame and second, to investigate the repercussions the coordinate transformation so described will have upon the form of the Minkowskian metric tensor in flat space. The first may readily be settled. Let us adopt cylindrical coordinates with the azimuth directed along the $z$-axis. The primary massive body in question here is the earth, so the the $x$- and $y$-axes will span the equatorial plane and the azimuth will coincide with the poles. In the earth's interior (that is, at radii $r<R_\oplus = 6371 ~\mathrm{km}$), the rotational velocity may be well approximated as constant, equal to $\omega_\oplus = 7.292115 \times 10^{-5} ~\mathrm{rad/s}$. In the earth's exterior ($r>R_\oplus$) we are free to choose any spatial dependence of the rotational velocity we please, subject to the following desiderata: the rotational velocity ought to decline to zero towards spatial infinity, and furthermore for the functional form to make sense it has to be constant around the perimeter of any circle lying in a plane parallel to the equator and centered on the azimuthal axis. Therefore, we may adopt a functional form purely for the sake of calculational convenience, understanding that at the end of the day one is always going to transform back to the non-rotating inertial frame established by the fixed stars. The following functional form will prove most expeditious:
\begin{equation}\label{rot_vel_fnctl_form}
	\omega = \begin{cases}
		\omega_\oplus & \mathrm{if}~ r<R_\oplus \\
		\omega_\oplus R_\oplus/r & \mathrm{if}~ r>R_\oplus. \\
	\end{cases}
\end{equation}
\begin{remark}
	Let us remark on the oddity that equation (\ref{rot_vel_fnctl_form}) implies a translational velocity relative to the fixed stars constant in magnitude as one proceeds toward spatial infinity within any plane parallel to the equatorial plane. But the problem is immaterial as far as we are concerned. We may merely suppose, implicitly, a faster fall-off with radius to set in at radii sufficiently far from the earth that any frame dragging effect there becomes negligible.
\end{remark}

Let now $x^{0,\ldots,3} = t, x, y, z$ denote Cartesian coordinates in the sidereal frame and $x^{\prime 0,\ldots,3} = t^\prime, x^\prime, y^\prime, z^\prime$ Cartesian coordinates in the co-rotating frame. They are related as follows:
\begin{equation}
	\begin{matrix}
		t^\prime = t & t = t^\prime \\
		x^\prime = x \cos \omega t - y \sin \omega t & x = x^\prime \cos \omega t^\prime + y^\prime \sin \omega t^\prime \\
		y^\prime = x \sin \omega t + y \cos \omega t & y = -x^\prime \sin \omega t^\prime + y^\prime \cos \omega t^\prime \\
		z^\prime = z & z = z^\prime \\
	\end{matrix}.
\end{equation}
In order to apply our formalism to the Minkowskian metric in co-rotating frame, let us consider first the case of an arbitrary change of orthonormal basis. To start with, we have with respect to the orthonormal co-frame basis $\varepsilon^{1,\ldots,n}$,
\begin{equation}
	\hat{\eta} = \sum_{|\gamma|=1}^\infty \sum_{\alpha,\beta\ge\gamma} \frac{|\alpha|!}{\alpha!} \frac{|\beta|!}{\beta!} (-1)^{\beta_0}\varepsilon^\alpha \otimes \varepsilon^\beta,
\end{equation}
cf. equation (\ref{generalized_minkowski_metric}). Now, in the cylindrical coordinates adapted to the problem the change of basis may simply be expressed as,
\begin{equation}
	\varepsilon^{t^\prime} = \varepsilon^t; \qquad \varepsilon^{\rho^\prime} = \varepsilon^\rho; \qquad \varepsilon^{\delta^\prime} = \varepsilon^\delta; ~ \mathrm{and}~ \qquad \varepsilon^{\phi^\prime} = \varepsilon^{\phi-\omega t},
\end{equation}
where $t$ is the time coordinate, $\rho$ the radial coordinate, $\delta$ the declination and $\phi$ the azimuth. Expanding up to first order in the rotational velocity should then leave $\varepsilon^{t^\prime}$, $\varepsilon^{\rho^\prime}$ and $\varepsilon^{\delta^\prime}$ unchanged and for $\varepsilon^{\phi^\prime}$ yield
\begin{equation}
	\varepsilon^{\phi^\prime} = \varepsilon^{\phi} - \omega_\oplus R_\oplus \cos \delta \varepsilon^{t^\prime}.
\end{equation}
Thus,
\begin{equation}\label{minkowski_metric_in_rotating_frame}
	\hat{\eta}^\prime = \hat{\eta} - \omega_\oplus R_\oplus \cos \delta \varepsilon^{t} \otimes \varepsilon^{\phi} +
	\omega_\oplus R_\oplus \cos \delta \varepsilon^t \otimes \varepsilon^{tt} + \cdots.
\end{equation}
Here, the perturbation in $\varepsilon^t \otimes \varepsilon^{tt}$ arises from $\varepsilon^{t^\prime} \otimes \varepsilon^{t^\prime\phi^\prime}$ resp. $\varepsilon^{\phi^\prime} \otimes \varepsilon^{t^\prime t^\prime}$ and the ellipsis comprehends all other such terms. Correspondingly, the Schwarzschild metric in the co-rotating frame is given by,
\begin{align}\label{schwarzschild_metric_in_rotating_frame}
	g^\prime = g - &
	\omega_\oplus R_\oplus \cos \delta d^t \otimes d^{\phi}
	\sqrt{1 - \frac{2GM_\oplus}{r}} + \nonumber \\
	&
	\omega_\oplus R_\oplus \cos \delta d^t \otimes d^{tt} \left( 1 - \frac{2GM_\oplus}{r} \right)
	+ \cdots.
\end{align}

The import of the modified form of the metric in the co-rotating frame is the following: if we examine the expression in equation (I.4.13) for the Christoffel symbols of the Levi-Civita jet connection in terms of the components of the metric tensor, we notice that its zero components $\Gamma^\mu_{0,0}$ resp. $\Gamma^\mu_{0,00}$ may be written as,
\begin{align}\label{christoffel_comparison}
	\Gamma^\mu_{0,0} &= \frac{1}{2} g^{\mu\lambda} \left( \partial_0 g_{0,\lambda} + \partial_{0} g_{\lambda 0} - \partial_\lambda g_{0,0} \right) \\
	\Gamma^\mu_{0,00} &= \frac{1}{2} g^{\mu\lambda} \left( \partial_0 g_{00,\lambda} + \partial_{00} g_{\lambda 0} - \partial_\lambda g_{0,00} \right).
\end{align}	
where we have suppressed cross-terms in view of their negligibility in the weak-field limit such as will hold near the earth. The other Christoffel symbols experiencing a non-zero perturbation, such as $\Gamma^\mu_{00,\nu}$ and $\Gamma^\mu_{0\nu,0}$, will be unimportant in the non-relativistic limit since they contract against components of the 2-velocity that are down by a factor of $\varv/c$. Let us further suppose a static space-time geometry (by which we shall mean simply that the components of the metric tensor have no explicit dependence on the time coordinate) and observe again that, in the weak-field limit, the factor of $g^{-1}$ in front may for all practical purposes be replaced with $\eta^{-1}$. Lastly, since for the time being we are interested in the dynamics applicable to spacecraft flying by earth at non-relativistic speeds, the dominant contribution to equation (\ref{christoffel_comparison}) will come from the $00$ respectively $00,0$ components; that is to say, we are left with
\begin{align}\label{christoffel_reduced}
	\Gamma^\mu_{0,0} &= - \frac{1}{2} \mathrm{grad}~ g_{00} \\
	\Gamma^\mu_{0,00} &= - \frac{1}{2} \mathrm{grad}~ g_{0,00}.
\end{align}
Thus, the formulae are precisely parallel and we may interpret this result as saying that the $g_{0,00}$ component of the metric tensor serves as another gravitational potential entering into the jet geodesic equation. For again, in the non-relativistic limit, the spacecraft's velocity vector to second order will be given to leading order simply by $\partial_0+2\partial_{00}$, up to terms of order $\varv/c$ (where $\varv$ denotes the ordinary 3-component translational velocity). Remembering now the jet geodesic equation $\nabla_X X=0$ written out in components:
\begin{equation}
	X^0 \dot{X}^\mu = - \Gamma^\mu_{0,0} X^0 X^0 - \Gamma^\mu_{0,00} X^0 X^{00} - \Gamma^\mu_{00,0} X^{00} X^0,
\end{equation}
we find (substituting $X^\mu=\gamma (1,\varv_1,\varv_2,\varv_3)$ with $\gamma := 1/\sqrt{1-\varv^2}$) we find the dominant contribution to the first-order part of the velocity to be (including a factor of two arising from $\Gamma^\mu_{00,0}$ respectively $\Gamma^\mu_{0,00}$)
\begin{equation}\label{nonrel_equ_of_motion}
	\dot{X}^k = \frac{1}{2} \gamma \mathrm{grad}_k g_{00} + 2 \gamma^2 \mathrm{grad}_k g_{0,00} \qquad (k=1,2,3).
\end{equation}
The first term on the right-hand side is just the force coming from the familiar Newtonian gravitational potential $g_{00} = - \left( 1 - 2GM/r \right)$ while the second term represents a novel force inherited from the 00-component of the velocity. Here, we have taken care to retain the Lorentz factors $\gamma$ respectively $\gamma^2$ (the exponent to the power of two in the latter arises from the fact that in the non-relativistic limit $X^{00} = 2 \gamma^2$).

Let us proceed to the derivation of the deviation term in the jet geodesic equation originating from the rotation of the earth. Here (going over now into spherical coordinates) the first-order part of the metric is given in the usual notation by
\begin{align}
	g &= - \left(1-2GM_\oplus/r \right) dt^2 + \left( 1 + 2GM_\oplus/r \right) dr^2 + r^2 d\delta^2 + r^2 \cos^2 \delta d\phi^2 \\
	g^{-1} &= - \left(1-2GM_\oplus/r \right)^{-1} \partial_t^2 + \left( 1 + 2GM_\oplus/r \right)^{-1} \partial_r^2 + r^2 \partial_\delta^2 + r^2 \cos^2 \delta \partial_\phi^2 \\ 
	&= - \left(1+2GM_\oplus/r \right) \partial_t^2 + \left( 1 - 2GM_\oplus/r \right) \partial_r^2 + r^2 \partial_\delta^2 + r^2 \cos^2 \delta \partial_\phi^2,
\end{align}
neglecting in the last line terms of order $(GM_\oplus/r)^2$. For the orthonormal frame in spherical coordinates we have
\begin{align}
	\hat{e}_r &= \left( 1 - \frac{2GM_\oplus}{r} \right) \partial_r \\
	\hat{e}_\delta &= \frac{1}{r} \partial_\delta \\
	\hat{e}_\phi &= \frac{1}{r \cos \delta} \partial_\phi
\end{align}
and the gradient in these coordinates is given by
\begin{equation}
	\mathrm{grad} = \hat{e}_r \frac{\partial}{\partial r} + \hat{e}_\delta \frac{1}{r} \frac{\partial}{\partial\delta} + \hat{e}_\phi \frac{1}{r \cos \delta} \frac{\partial}{\partial\phi}.
\end{equation}
We want to take for the additional potential, as above,
\begin{equation}
	g_{0,00} = - \omega_\oplus R_\oplus \cos \delta \left( 1 - \frac{2GM_\oplus}{r} \right).
\end{equation}
Therefore,
\begin{align}
	2 ~\mathrm{grad}~ g_{0,00} &= \frac{2\omega_\oplus R_\oplus \sin \delta}{r} \hat{e}_\delta - \frac{4\omega_\oplus R_\oplus}{r^2} GM_\oplus \sin \delta ~\hat{e}_\delta -  \frac{4\omega_\oplus R_\oplus}{r^2} GM_\oplus \cos \delta ~\hat{e}_r \nonumber \\
	&= - 2\omega_\oplus R_\oplus \left[ \frac{2GM_\oplus}{r^2} \cos \delta ~\hat{e}_r -
	\left( \frac{1}{r} - \frac{2GM_\oplus}{r^2} \right) \sin \delta ~\hat{e}_\delta \right].
\end{align}
But we must bear in mind that this would yield the force to be experienced by a spacecraft at rest. To obtain the force on a spacecraft moving at velocity $\varv \ll c$ we must make two corrections: first, multiply by $\gamma^2=1+\varv^2$ to leading order leading to,
\begin{equation}\label{force_in_proper_time}
	F =  - 2\omega_\oplus R_\oplus \left[ \frac{2GM_\oplus}{r^2} \cos \delta ~\hat{e}_r -
	\left( \frac{1+\varv^2}{r} - \frac{2GM_\oplus}{r^2} \right) \sin \delta ~\hat{e}_\delta \right],
\end{equation}
where of course we may drop as negligible terms proportional to $GM_\oplus \varv^2/rc^2$. The second correction is of the essence here: we must take into account the modification to the local standard of time introduced by the fact that the spacecraft is not moving in free space but in outer space near to a massive body. As Einstein discovered, this modification leads to the so-called gravitational red-shift for spectral lines emanating from elements in stellar atmospheres. Here, the same effect causes a small but crucial shift in the spacecraft's measure of proper time: namely, we now have $ds = \left( 1 - 2GM_\oplus/r \right) dt$, whence
\begin{equation}
	\frac{d}{ds} = \frac{1}{1-2GM_\oplus/r} \frac{d}{dt}.
\end{equation}
Remember now that the although we are in principle free to write the jet geodesic equation in arbitrary coordinates, from a physical point of view it is appropriate to employ the moving body's own proper time when computing the rate of change of quantities it experiences in its co-moving local frame; viz. $ds$ rather than in terms of the time as measured by stationary clocks at spatial infinity; viz., $dt$. Keep in mind here that, in the relevant contribution, the left-hand side involves $d^2/dt^2$ of the coordinate functions whereas the right-hand side involves, instead, a factor of $d^2/dt^2 d/dt$ applied to the Christoffel symbol that connects the $\partial_{00}$ component of the velocity with the $\partial_0$ component. When we convert from $d/dt$ to $d/ds$, then, we bring in a factor of $(1-2GM_\oplus/r)^{2}$ on the left-hand side which balances against a factor of $(1-2GM_\oplus/r)^{3}$ on the right-hand side to produce, net, a factor of $(1-2GM_\oplus/r)$ on the right-hand side.

The upshot of the above rather subtle discussion is that we are to multiply the force in equation (\ref{force_in_proper_time}) by a factor of $(1-2GM_\oplus/r)$ and retain only leading terms; hence, after the correction for proper time the force law reads,
\begin{equation}
	F = - 2\omega_\oplus R_\oplus \left[ \frac{2GM_\oplus}{r^2} \cos \delta ~\hat{e}_r -
	\left( \frac{1+\varv^2}{r} - \frac{4GM_\oplus}{r^2} \right) \sin \delta ~\hat{e}_\delta \right],
\end{equation}

We turn now to the inference from the derived non-relativistic equation of motion to the flyby anomaly, i.e., derivation of the differential between outgoing and ingoing asymptotic velocities, making appeal to the arguments found in unpublished work by W. Hasse, E. Birsin and P. H{\"a}hnel, On force-field models of the spacecraft flyby anomaly, arXiv:0903.0109v1 [gr-qc] (February 28, 2009). First, in {\S}2 of their paper, they show that the velocity-independent term proportional to $1/r$ makes no contribution to the flyby anomaly, so we may drop this term. Second, in {\S}3 they prove that one is free to add any force term of the form
\begin{equation}
	F_w = - 2\omega_\oplus R_\oplus \left[ w \frac{2GM_\oplus}{r^2} \cos \delta ~\hat{e}_r + w \frac{2GM_\oplus}{r^2} \sin \delta ~\hat{e}_\delta \right]
\end{equation}
without affecting the final result for the flyby anomaly, where $w$ is an arbitrary real number. Thus, if we set the parameter $w$ to $-2$ we obtain the force law that generates the flyby anomaly in the simple functional form,
\begin{equation}\label{final_force_law}
	F = 2\omega_\oplus R_\oplus \left[ \frac{2GM_\oplus}{r^2} \cos \delta ~\hat{e}_r +
	\frac{\varv^2}{r} \sin \delta ~\hat{e}_\delta \right].
\end{equation}
As Hasse et al. argue, a force law of this kind (note, our equation (\ref{final_force_law}) reproduces Hasse's et al.'s solution $F_0$ in their notation, Equation 3.24 in the cited paper) will in fact lead upon a perturbational calculation precisely to 
Anderson et al.'s semi-empirical prediction formula for the flyby anomaly \cite{anderson_flyby}, to wit:
\begin{equation}\label{anderson_prediction_formula}
	\frac{\Delta\varv_\infty}{\varv_\infty} = 2 \omega_\oplus R_\oplus \left( \cos \delta_0 - \cos \delta_1 \right),
\end{equation}
for the differential $\Delta\varv_\infty$ between the outgoing asymptotic velocity $\varv_1$ and the incoming asymptotic velocity $\varv_0$. Here, $\delta_{0,1}$ designate the declination of the planes in which the incoming respectively outgoing asymptotic velocities vectors lie (i.e., passing through the spacecraft and the center of mass of the earth). 

\begin{proposition}
The force law in equation (\ref{final_force_law}) leads to the flyby anomaly in the analytical form of equation (\ref{anderson_prediction_formula}).
\end{proposition}
\begin{proof}
We recapitulate here the main steps in the derivation, originally due to Hasse et al. (op. cit.). First of all, note that the velocity differential $\Delta \varv_\infty$ may be written in terms of the spacecraft's internal energy per unit mass as
\begin{equation}\label{velocity_differential}
\frac{\Delta \varv_\infty}{\varv_\infty} = \frac{\Delta E}{2E},
\end{equation}
where $E = \dfrac{1}{2}\varv^2 - \dfrac{GM_\oplus}{r}$. To a leading approximation, the change in internal energy may be obtained by integrating the perturbing force over the entire trajectory; thus,
\begin{align}
\Delta E &= \int_{-\infty}^{\infty} \mathbf{F} \cdot \mathbf{v} dt \nonumber \\
&= 2 \omega_\oplus R_\oplus  \int_{-\infty}^{\infty} \left[ \frac{2GM_\oplus}{r^2} \cos \delta ~\hat{e}_r + \frac{\varv^2}{r} \sin \delta ~\hat{e}_\delta \right] \cdot \mathbf{v} dt \nonumber \\
&= 2 \omega_\oplus R_\oplus \int_{-\infty}^{\infty} \left[ \frac{\varv^2}{r} - \frac{2GM_\oplus}{r^2} \right] \sin \delta ~\hat{e}_\delta \cdot \mathbf{v} dt \nonumber \\
&= - 2\omega_\oplus R_\oplus \int_{-\infty}^{\infty} \left[ \varv^2 - \frac{2GM_\oplus}{r} \right] \nabla \cos \delta \cdot \mathbf{v} dt \nonumber \\
&= - 2 \omega_\oplus R_\oplus \int_{-\infty}^{\infty} \left[ \varv^2 - \frac{2GM_\oplus}{r} \right] \frac{d \cos \delta}{dt} dt \nonumber \\
& = 4 \omega_\oplus R_\oplus E \left( \cos \delta_0 - \cos \delta_1 \right).
\end{align}
To go from the second to third lines, use the fact that we are free to add any term proportional to 
\begin{equation}
\nabla \dfrac{\cos \delta}{r} = - \dfrac{\cos \delta}{r^2} \hat{e}_r - \dfrac{\sin \delta}{r^2} \hat{e}_\delta,
\end{equation}
as its integral will not contribute to $\Delta E$. To go from the penultimate to the last line, note that $E$ may be taken outside the integral since it is approximately conserved (neglecting the small perturbation). In view of equation (\ref{velocity_differential}) we obtain the semi-empirical prediction formula of equation (\ref{anderson_prediction_formula}).
\end{proof}

\subsubsection{Experimental Status of the Flyby Anomaly}

For an experimental test and its adequacy, see the paper by Anderson et al., \cite{anderson_flyby}, whose data are reproduced in Table 9.1. For flybys of the earth conducted as of the time of writing, satisfactory agreement is found (some other more recent flybys have been inauspicious for observing the anomaly since they occur with declinations $\delta_0$ and $\delta_1$ such as to produce a small effect in any case). Some of the figures for the precision quoted by these authors seem unrealistically small and would lead to absurdly high $\chi^2$. Therefore, we have adopted the expedient of averaging to arrive at $\bar{\sigma}_{\Delta \varv_\infty} = 0.39$, from which we find a $\chi^2$ per d.o.f. of 1.11 and a p-value of 0.019. All things considered, in view of the difficulty of performing an experiment on relativistic effects, one is entitled to presume the flyby anomaly to be real and Anderson's prediction formula to embody it (na{\"i}vely the observed and predicted $\Delta \varv_\infty$ have a correlation coefficient of 0.9983, which could hardly be due to chance alone). In particular, it is noteworthy that the observed $\Delta \varv_\infty$ displays little apparent connection with the altitude of the spacecraft at perigee (correlation coefficient of just $-0.177$).

\begin{table}[h!]
	\begin{center}
		\footnotesize	
		\caption{Anomalous differential between outgoing and ingoing asymptotic velocities for flybys of the earth, data reproduced from \cite{anderson_flyby}.}
		\begin{tabular}{c c d{4.0} d{2.3} d{3.2} d{3.2} d{2.2} d{1.2} d{2.2} d{1.2}}
			\noalign{\vskip 3mm}
			\hline\hline
			\noalign{\vskip 1mm}
			Date & Spacecraft &	\multicolumn{1}{c}{Altitude} & \multicolumn{1}{c}{$\varv_\infty$} & \multicolumn{1}{c}{$\delta_0$} & \multicolumn{1}{c}{$\delta_1$} & \multicolumn{1}{c}{Observed} & \multicolumn{1}{c}{$\sigma_{\Delta \varv_\infty}$} & \multicolumn{1}{c}{Predicted} & \multicolumn{1}{c}{$\chi^2$} \cr
			& & \multicolumn{1}{c}{(km)} & \multicolumn{1}{c}{(km/s)} & \multicolumn{1}{c}{(${}^\circ$)} & \multicolumn{1}{c}{(${}^\circ$)} & \multicolumn{1}{c}{$\Delta \varv_\infty$ (mm/s)} & \multicolumn{1}{c}{(mm/s)} & \multicolumn{1}{c}{$\Delta \varv_\infty$ (mm/s)} & \cr
			\noalign{\vskip 1mm}
			\hline
			\noalign{\vskip 2mm}
			12/08/90 & GLL-I & 960 & 8.949 & -12.52	& -34.15 & 3.92	& 0.3 &	4.12	& 0.27 \cr
			12/08/92 & GLL-II & 303 & 8.877 & -34.26 & -4.87 & -4.6	& 1	& -4.67 &	0.04 \cr
			01/23/98 & NEAR & 539 & 6.851 & -20.76 & -71.96	& 13.46	& 0.01 & 13.28 & 0.21 \cr
			08/18/99 & Cassini & 1175 & 16.01 & -12.92 & -4.99 & -2 & 1 & -1.07 & 5.66 \cr
			03/04/05 & Rosetta & 1956 & 3.863 & -2.81 & -34.29 & 1.8 & 0.03 & 2.07 & 0.46 \cr
			08/02/05 & M’GER & 2347 & 4.056 & 31.44 & -31.92 & 0.02	& 0.01 & 0.06 & 0.01 \cr
			\noalign{\vskip 2mm}
			\hline\hline
		\end{tabular}	
	\end{center}	
\end{table}

\subsubsection{Another Interpretation of the Physics behind the Flyby Anomaly}

It also bears pointing out that despite the indirect involvement of gravitation through the magnitude of the $g_{0,0\phi}$ components of the metric tensor, aetiologically the flyby anomaly is primarily kinematical in origin. Another interpretation of the effect would be to construe the flyby anomaly as a contribution to redshift. To determine the local standard of time, one has to take into account not just the gravitational modification to $g_{00}$ but also the distance traveled by a stationary clock in the co-moving frame, as can be computed by direct integration:
\begin{proposition}
	The arclength covered by stationary clocks in the co-moving frame is given up to second order by $\lambda = t + \frac{1}{2}t^2 - \omega_\oplus R_\oplus t^2 \cos \delta$.
\end{proposition}
\begin{proof}
	The first correction term will be of second order in $\omega_\oplus R_\oplus$ and may be neglected, as is conventional in the general theory of relativity. It is easiest to write the second correction term as proportional to distance traveled, which at radius $r$ and declination $\delta$ will be just $d^s = r \cos \delta d^\phi = \omega r \cos \delta d^t = \omega_\oplus R_\oplus \cos \delta d^t$. 
	
	Clearly with $d^\lambda = d^t + d^s$ we pick up another term in $d^{\lambda\lambda} = d^{tt} + d^{ss} + 2 d^{ts}$. In the metric tensor, for which we may employ the Minkowskian form since any term in the Schwarzschild metric in $GM_\oplus/r$ will be vanishingly small when multiplied together with $\omega_\oplus R_\oplus$, the temporal-spatial components in the 22-sector have a relative minus sign. Thus, the integrand for arclength will become $d^t+\left( 1 - 2 \omega_\oplus R_\oplus \cos \delta \right) d^{tt}$ in place of the usual $d^t + d^{tt}$, yielding the desired result upon integration.
\end{proof}
If we go from the primed coordinate system back to the unprimed, the sign of the second correction term changes. The difference in $d^{tt}$ must be entered into the metric tensor, and thereby affects $g_{0\phi,0}$ as indicated above. Of course to get the spacecraft's proper time one would have to multiply by its Lorentzian factor. The second correction term is very small since to non-dimensionalize we should multiply $t^2$ by a factor of $c^2/b^2$ and the fundamental length is presumably astronomically large compared to terrestrial or solar system scales; it would not have been detected if it were not for the availability of an extremely precise standard of time.

\begin{remark}
A correct prediction of the flyby anomaly indirectly confirms our principle of plenitude in much the same way as Kuhn and Thomas' sum rule \cite{kuhn}, \cite{thomas} provides an experimental test of Jordan's commutator relations \cite{born_jordan} between the position and momentum q-numbers in non-relativistic quantum mechanics.
\end{remark}

	\section{Discussion}
	
The rightness of our paradigm is confirmed by its success in predicting new phenomena in orbital mechanics. The proposed resolution of the Pioneer and flyby anomalies indicates that the precise tracking of spacecaft in the solar system can serve as a window through which to peer into the world of infinitesimals. The result calls into question whether space-time can adequately be modeled on a differentiable manifold as this term is ordinarily understood in twentieth-century mathematics. Modern mathematicians are conversant with far more versatile abstractions of the notion of space allowing for singularities to be handled unproblematically, such as are encountered in algebraic geometry (varieties, possibly non-reduced schemes, stacks etc.) and elsewhere, not least in the Russian school of differential geometry (orbifolds, diffeologies, $C^{\infty}$-ringed spaces etc.). The present contribution should be seen in light of this ongoing trend in the recent mathematical literature to introduce enlarged concepts of space, the distinctive properties of which are fitted to expand the range of applicability of spatial intuition. What is salient about it is twofold: first, unlike what is the case in algebraic geometry as normally pursued, one works in the context of smooth functions so that unlimited differentiability holds and therefore one can avail oneself of the qualitative continuity properties that most naturally accord with spatial intuition; and second, through the postulation of a Riemannian metric tensor in a sense generalized to include infinitesimals of higher than first order, it makes the powerful geometrical concepts of orthogonality and the inner product between two vectors available to the more general setting and renders the problem susceptible to analysis through quantitative means. 

If it were not for these two things, no matter how vivid one's faculty of imagination may be in visualizing space, it would have led to no issue and one could not have hit upon the jet connection or the generalized Riemannian curvature, which, as we have seen, are what unlock innovative ideas in fundamental physics, which have received here but the merest sketch and which promise far-reaching advances in the future. Clearly, without an intuitively appealing (`anschauliches') picture of what a higher-order tangent vector is---which depends crucially on the assumed smoothness, of course---, one could not have gotten anywhere.

Hence, there is a strong subsequent case to be made for the proposition that infinitesimals and higher-order tangents ought to play a role in the foundations of geometry and, so one would expect, of mechanics as well. What is novel in the point of view advanced herein is that these infinitesimals are to be conceived of as really existing in nature. How are we to judge a statement such as this? In scheme theory, infinitesimals serve as a formal device by means of which to describe certain limiting situations and, absolutely speaking, would not be needed except for the sake of convenience. From the physicist's standpoint, the limiting situation in question is one where the typical distances between bodies are small in comparison to the cosmological length scale. The theory lends itself to an expansion in powers of this ratio, where the leading terms must dominate. As we have seen in {\S}\ref{chapter_7}, however, sufficiently sensitive observations may disclose effects at subleading order. The mooted idealization, therefore, suggests an analytical procedure to study the problem of motion of bodies order-by-order in the infinitesimals. In Part III, we shall turn to the implications of a formal procedure of this kind for a reconceptualization of the nature of the fundamental forces themselves.


\begin{thebibliography}{99}
	
\bibitem{hendry} J. Hendry, \textit{James Clerk Maxwell and the Theory of the Electromagnetic Field}, Adam Hilger Ltd., Bristol and Boston (1986).

\bibitem{maxwell_treatise} J.C. Maxwell, \textit{A Treatise on Electricity \& Magnetism}, vols. 1-2 (unabridged third edition), Dover, New York (1954).

\bibitem{peebles_ratra}  P.J.E. Peebles and B. Ratra, The cosmological constant and dark energy, \textit{Rev. Mod. Phys.} \textbf{75}, 559-606 (2003).
 
\bibitem{frieman_turner_huterer} J.A. Frieman, M.S. Turner and D. Huterer, Dark energy and the accelerating universe, \textit{Ann. Rev. Astron. Astrophys.} \textbf{46}, 385-432 (2008).

\bibitem{einstein_grundlage} A. Einstein, Die Grundlage der allgemeinen Relativit{\"a}tstheorie, \textit{Ann. Phys.} \textbf{49}, 769-822 (1916).

\bibitem{einstein_mach_principle} A. Einstein, Prinzipielles zur allgemeinen Relativit{\"a}tstheorie, \textit{Ann. Phys.} \textbf{55(4)}, 241-244 (1918).

\bibitem{norton_physical_content_gr} J. Norton, Physical content of general covariance, in J. Eisenstaedt and A.J. Knox, eds., \textit{Studies in the History of General Relativity}, based on proceedings of the second International Conference on the History of General Relativity, Luminy, France, 1988, Birkh{\"a}user (1992), pp. 281-315.

\bibitem{thirring_vol_2} W. Thirring, \textit{Classical Field Theory}, vol. 2 in \textit{A Course in Mathematical Physics} (first edition), Springer-Verlag (1979).

\bibitem{wald} R.M. Wald, \textit{General Relativity}, University of Chicago Press (1984).

\bibitem{schwarzschild} K. Schwarzschild, {\"U}ber das Gravitationsfeld eines Massenpunktes nach der Einsteinschen Theorie, \textit{Sitzungsberichte der K{\"o}niglich Preu{\ss}ischen Akademie der Wissenschaften}, 189-196, Berlin (1916).

\bibitem{friedmann} A. Friedmann, {\"U}ber die Kr{\"u}mmung des Raumes, \textit{Z. Phys.} \textbf{10}, 377-386 (1922); {\"U}ber die M{\"o}glichkeiten einer Welt mit konstanter negativer Kr{\"u}mmung des Raumes, ibid. \textbf{21}, 326-332 (1924).

\bibitem{robertson} H.P. Robertson, Kinematics and world structure I, \textit{Ap. J.} \textbf{82}, 284-301 (1935); Kinematics and world structure II, ibid. \textbf{83}, 187-201 (1936); Kinematics and world structure III, ibid. \textbf{83}, 257-271 (1936).

\bibitem{walker} A.G. Walker, On the Milne’s theory of world-structure, \textit{Proc. Lond. Math. Soc.} \textbf{2(1)}, 90-127 (1937).

\bibitem{reiss_et_al} A.G. Riess et al., Observational evidence from supernovae for an accelerating universe and a cosmological constant, \textit{A. J.} \textbf{116}, 1009-1038 (1998).

\bibitem{perlmutter_et_al} S. Perlmutter et al., Measurement of $\Omega$ and $\Lambda$ from 42 high-redshift supernovae, \textit{A. J.} \textbf{517}, 565-586 (1999).

\bibitem{shah_lemos_lahav} P. Shah, P. Lemos and O. Lahav, A buyer's guide to the Hubble constant, \textit{Astron. Astrophys. Rev.} \textbf{29(1)}, 1-69 (2021).

\bibitem{anderson_Pioneer} J.D. Anderson, P.A. Laing, E.L. Lau, A.S. Liu, M.M. Nieto and S.G. Turyshev, Study of the anomalous acceleration of Pioneer 10 and 11, {\it Phys. Rev. D}, \textbf{65(8)}, 082004 (2002).

\bibitem{turyshev} S.G. Turyshev, V.T. Toth, G. Kinsella, S.-C. Lee, S.M. Lok and J. Ellis, Support for the thermal origin of the Pioneer anomaly, \textit{Phys. Rev. Lett.} \textbf{108}, 241101 (2012).

\bibitem{thirring_vol_1} W. Thirring, {\it Classical Dynamical Systems}, vol. 1 in {\it A Course in Mathematical Physics} (second edition), Springer-Verlag (1978, 1992).

\bibitem{milgrom} M. Milgrom, A modification of the Newtonian dynamics as a possible alternative to the hidden mass hypothesis, \textit{Ap. J.} \textbf{270}, 365-370 (1983); A modification of the Newtonian dynamics - Implications for galaxies, \textit{Ap. J.} \textbf{270}, 371-383 (1983); A modification of the Newtonian dynamics - Implications for galaxy systems, \textit{Ap. J.} \textbf{270}, 384-389 (1983); Dynamics with a nonstandard inertia-acceleration relation: An alternative to dark matter in galactic systems, \textit{Ann. Phys.} \textbf{229(2)}, 384-415 (1994).

\bibitem{mond_review} M. Milgrom, MOND vs. dark matter in light of historical parallels, \textit{Stud. Hist. Philos. Sci. B Stud. Hist. Philos. Modern Phys.} \textbf{71}, 170-195 (2020).

\bibitem{anderson_flyby} J.D. Anderson, J.K. Campbell, J.E. Ekelund, J.E. Ellis and J.E. Jordan,  Anomalous orbital-energy changes observed during spacecraft flybys of Earth, {\it Phys. Rev. Lett.}, \textbf{100}, 091102 (2008).

\bibitem{einstein_formale_grundlage} A. Einstein, Die formale Grundlage der allgemeinen Relativit{\"a}tstheorie, \textit{Sitzungsberichte der Preussischen Akademie der Wissenschaften} \textbf{2}, 1030-1085 (1914).

\bibitem{lense_thirring} See H. Thirring, {\"U}ber die Wirkung rotierender ferner Massen in der Einsteinschen Gravitationstheorie, \textit{Phys. Z.} \textbf{19}, 33-39 (1918); J. Lense and H. Thirring, {\"U}ber den Einflu{\ss} der Eigenrotation der Zentralk{\"o}rper auf die Bewegung der Planeten und Monde nach der Einsteinschen Gravitationstheorie, \textit{Phys. Z.} \textbf{19}, 156-163 (1918).

\bibitem{kuhn} W. Kuhn, {\"U}ber die Gesamtst{\"a}rke der von einem Zustande ausgehenden Absorptionslinien, \textit{Z. Phys.} \textbf{33(1)}, 408-412 (1925).

\bibitem{thomas} W. Thomas, {\"U}ber die Zahl der Dispersionselektronen, die einem station{\"a}ren Zustande zugeordnet sind (Vorl{\"a}ufige Mitteilung), \textit{Naturwissenschaften} \textbf{13(28)}, 627 (1925). 

\bibitem{born_jordan} M. Born and P. Jordan, Zur Quantenmechanik, \textit{Z. Phys.} \textbf{34(1)}, 858-888 (1925).

\end{thebibliography}
\end{document}